\documentclass[reqno]{amsart}
\usepackage{amssymb,amsmath,hyperref}
\usepackage{amsrefs}
\usepackage[foot]{amsaddr}
\usepackage{bbold,stackrel}
 
\newtheorem{thm}{Theorem}

\newtheorem{cor}[thm]{Corollary}
\newtheorem{defi}[thm]{Definition}
\newtheorem{rem}[thm]{Remark}
\newtheorem{nota}[thm]{Notation}

\newtheorem{ack}[thm]{Acknowledgement}

\newtheorem*{tempo*}{Template}

\newcommand\be{\begin{equation}}
\newcommand\ee{\end{equation}} 

\usepackage{amsmath,amsfonts} 

\usepackage[applemac]{inputenc}

\def\bdefi{\begin{defi}\rm}
\def\edefi{\end{defi}}
\def\bnota{\begin{nota}\rm}
\def\enota{\end{nota}}
\def\FIVE{\Pi_{1}^{1}\text{-\textup{\textsf{CA}}}_{0}}

\def\SIXK{\Pi_{k}^{1}\text{-\textsf{\textup{CA}}}_{0}^{\omega}}

\def\Z{\textup{\textsf{Z}}}

\def\NFP{\textup{\textsf{NFP}}}

\def\ZF{\textup{\textsf{ZF}}}

 \def\r{\mathbb{r}}

\def\c{\textup{\textsf{c}}}
\def\RCA{\textup{\textsf{RCA}}}
\def\({\textup{(}}
\def\){\textup{)}}

\def\RCAo{\textup{\textsf{RCA}}_{0}^{\omega}}
\def\ACAo{\textup{\textsf{ACA}}_{0}^{\omega}}

\def\WKL{\textup{\textsf{WKL}}}

\def\bye{\end{document}}
\def\N{{\mathbb  N}}
\def\Q{{\mathbb  Q}}
\def\R{{\mathbb  R}}
\def\L{\textsf{\textup{L}}}

\def\di{\rightarrow}
\def\asa{\leftrightarrow}
\def\ACA{\textup{\textsf{ACA}}}

\def\QFAC{\textup{\textsf{QF-AC}}}

\def\HBU{\textup{\textsf{HBU}}}

\def\AUTO{\textup{\textsf{AUTO}}}

\def\UACL{\textup{\textsf{UACL}}}
\def\ALCL{\textup{\textsf{ALCL}}}

\def\CA{\textup{\textsf{CA}}}

\def\SEP{\textup{\textsf{SEP}}}
\def\rado{\textup{\textsf{Rado}}}

\def\range{\textup{\textsf{range}}}
\def\RANGE{\textup{\textsf{RANGE}}}

\def\SS{\textup{\textsf{S}}}
\def\IND{\textup{\textsf{IND}}}
\def\net{\textup{\textsf{net}}}
\def\CLO{\textup{\textsf{CLO}}}

\def\BOOT{\textup{\textsf{BOOT}}}

\def\CLO{\textup{\textsf{CLO}}}

\def\seq{\textup{\textsf{seq}}}

\def\ORD{\textup{\textsf{ORD}}}

\def\MCT{\textup{\textsf{MCT}}}

\def\ECF{\textup{\textsf{ECF}}}

\usepackage{graphicx}
\usepackage{tikz}
\usetikzlibrary{matrix, shapes.misc}

\numberwithin{equation}{section}
\numberwithin{thm}{section}

\usepackage{comment}

\begin{document}
\title{Lifting proofs from countable to uncountable mathematics}
\author{Sam Sanders}
\address{Department of Mathematics, TU Darmstadt, Germany}
\email{sasander@me.com}
\begin{abstract}
Turing's famous `machine' framework provides an intuitively clear conception of `computing with real numbers'.  
A \emph{recursive counterexample} to a theorem shows that the theorem does not hold when restricted to computable objects. 
These counterexamples are often crucial in establishing \emph{reversals} in the \emph{Reverse Mathematics} program.  
All the previous is essentially limited to a language that can only express \emph{countable} mathematics directly.  
The aim of this paper is to show that reversals and recursive counterexamples, countable in nature as they might be, directly yield new and interesting 
results about \emph{uncountable} mathematics \emph{with little-to-no modification}.  
We shall treat the following topics/theorems: the monotone convergence theorem/Specker sequences, compact and closed sets in metric spaces, the Rado selection lemma, the ordering and algebraic closures of fields, and ideals of rings.  The higher-order generalisation of sequence is of course provided by \emph{nets} (aka \emph{Moore-Smith sequences}).   
\end{abstract}


\maketitle
\thispagestyle{empty}

\section{Introduction}\label{intro}
In a nutshell, the aim of this paper is to show that a certain class of proofs originating from \emph{computability theory} about (essentially) countable mathematics gives rise to new and interesting proofs in \emph{uncountable} mathematics \emph{with very little modification}.  An informal description can be found in Section \ref{kintro} while a more technically detailed sketch is in Section \ref{farg}.

\subsection{Informal motivation and summary}\label{kintro}
An ever-returning question in mathematics is whether a proof of a given theorem gives rise to a more general result, perhaps after some modification to the proof.  
For instance, the (inherently countable) notion of `sequence of real numbers' has been generalised (to the uncountable) in the form of the topological notion of `net', introduced below.   
It is then a natural question whether proofs of theorems about sequences (somehow) give rise to proofs of theorems about nets.
More generally, is there a systematic way of transferring {proofs} about countable mathematics to proofs about uncountable mathematics?
The aim of this paper is to provide many examples of proofs that transfer \emph{directly} from countable to uncountable mathematics \emph{with very little modification}.
Our results \emph{suggest} the existence of a systematic procedure, which we are yet to discover.  

\smallskip

To illustrate the nature of the aforementioned transfer, an example from undergraduate calculus may be illuminating at this point.  To this end, consider the \emph{monotone convergence theorem} ($\MCT_{\seq}^{[0,1]}$ for short) which expresses that 
\emph{a monotone sequence in the unit interval converges}.  It is well-known that one can derive \emph{arithmetical comprehension} ($\ACA_{0}$ for short) from $\MCT_{\seq}^{[0,1]}$, where the former comprehension axiom implies a solution to Turing's \emph{Halting problem} (see e.g.\ \cite{simpson2}*{III.2.2}).  
We show below that \emph{with very little modification}, this well-known proof of $\MCT_{\seq}^{[0,1]}\di \ACA_{0}$ yields a proof of $\MCT_{\net}^{[0,1]}\di \BOOT$; here $\MCT_{\net}^{[0,1]}$ is a version of 
the monotone convergence theorem \emph{for nets} in the unit interval while $\BOOT$ is a very strong (uncountable) comprehension axiom dwarfing $\ACA_{0}$.  

\smallskip

The previous paragraph can be summarised as follows: some results on sequences (countable as they may be) can be \emph{directly} translated to a much more general result about 
nets, which intuitively are the generalisation of sequences to the uncountable.  In this paper, we establish many such results in which proofs about countable mathematics are `transferred' or `lifted' to proofs about uncountable mathematics.      

\smallskip

As discussed in detail in Section \ref{farg}, we mostly study proofs that originate in \emph{computability theory}.  
These proofs pertain to the following topics/theorems: the monotone convergence theorem/Specker sequences, compact and closed sets in metric spaces, the Rado selection lemma, the ordering and algebraic closures of fields, and ideals of rings.  The higher-order generalisation of `sequence' is of course provided by \emph{nets} (aka \emph{Moore-Smith sequences}), as suggested above.  

\smallskip

Finally, this paper is a spin-off from my joint project with Dag Normann on the Reverse Mathematics and computability theory of the uncountable.  
The interested reader should consult \cite{dagsamIII, dagsamV, dagsamVII} for a published and detailed account of our project. 

%

\subsection{Introduction and summary}\label{farg}
In this section, we make the informal summary from Section~\ref{kintro} more precise.  
In particular, we shall select a particular class of proofs about `countable mathematics' while also making the latter concept precise.  
We provide a list of theorems of which the proofs will be lifted to `uncountable mathematics' in Section \ref{main}, while also making the latter concept precise. 
An overview of the technical prerequisites of this paper is given in Section \ref{prelim}.

\smallskip

First of all, \emph{computability theory} has its roots in the seminal work of Turing, providing an intuitive notion of computation based on what we nowadays call \emph{Turing machines} (\cite{tur37}).
Now, \emph{classical} (resp.\ higher-order) computability theory deals with the computability of real numbers (resp.\ higher-order objects).  
In classical computability theory, a \emph{recursive counterexample} to a theorem (formulated in an appropriate language) shows that the latter does not hold when restricted to computable sets. 
An historical overview may be found in the introduction of \cite{recmath1}.  

\smallskip

Recursive counterexamples turn out to be highly useful in the Friedman-Simpson \emph{Reverse Mathematics} program ({RM} hereafter; see Section \ref{prelim1}).  
Indeed, the aim of RM is to determine the minimal axioms needed to prove a given theorem of ordinary\footnote{Simpson describes \emph{ordinary mathematics} in \cite{simpson2}*{I.1} as \emph{that body of mathematics that is prior to or independent of the introduction of abstract set theoretic concepts}.} mathematics, often resulting in an \emph{equivalence} between these axioms and the theorem;
recursive counterexamples often (help) establish the `reverse' implication (or: reversal) from the theorem at hand to the minimal axioms.  We refer to \cite{rauwziel}*{p.\ 1368} for this opinion, while a number of recursive counterexamples are mentioned in \cite{simpson2} and the RM literature at large.  

\smallskip

As is well-known, both (classical) RM and recursion theory are essentially restricted to $\L_{2}$, the language of second-order arithmetic, which only has variables for natural numbers and sets thereof; uncountable objects are (only) available in $\L_{2}$ via countable representations.  A good example is that of (codes for) continuous functions as given in e.g.\ \cite{simpson2}*{II.6.1}.  
Historically, second-order arithmetic stems from the logical system $H$ in the \emph{Grundlagen der Mathematik} (see \cite{hillebilly2, hillebilly}) in which Hilbert and Bernays formalise large swaths of mathematics.  While the system $H$ involves third-order parameters and operators, Hilbert and Bernays sketch alternative systems without the latter. 

\smallskip

In the previous paragraphs, we observed that the language second-order arithmetic is (essentially) restricted to countable objects and we identified a particular class of proofs 
in this language, namely `reversals' and `recursive counterexamples'. 
The aim of this paper is now to show that \emph{with very little modification} recursive counterexamples from computability theory and reversals from RM, second-order/countable as they may be, can yield interesting results in higher-order mathematics.  
In particular, we shall treat the following theorems/topics in this way.  
\begin{enumerate}
 \renewcommand{\theenumi}{\alph{enumi}}
\item montone convergence theorem/Specker sequences (Section \ref{specknets}), \label{gong}
\item compactness of metric spaces (Section \ref{metricsect}), \label{gong2}
\item closed sets in metric spaces (Section \ref{clookes}), 
\item Rado selection lemma (Section \ref{radokes}), \label{gong3}
\item ordering of fields (Section \ref{forder}), \label{gong4}
\item algebraic closures of fields (Section \ref{aclo}),\label{killoo}
\item ideals of rings (Section \ref{maxiprime}).
\end{enumerate}
The example involving item \eqref{gong} was sketched in Section \ref{kintro}.   
As another example, item \eqref{gong4} pertains to Ershov's recursive counterexample about the ordering of \emph{countable} fields, which we will lift to a similar result about \emph{uncountable} fields.  

\smallskip

%
%
We note that items \eqref{gong}, \eqref{gong2}, \eqref{gong3}, and \eqref{gong4}, were first published in CS proceedings, namely \cite{samflo2}.
We shall work in the formalism provided by Kohlenbach's \emph{higher-order} RM (\cite{kohlenbach2}; see Section \ref{prelim1} for an introduction).  
To be absolutely clear, our `lifted results' often require an extension of Kohlenbach's base theory with either a fragment of countable choice or extra comprehension; 
by contrast, known proofs of the same result (involving no lifting) go through in the aforementioned base theory.  We provide an example in Section \ref{specknetssub}.

\smallskip

We stress that no systematic procedure is given (or is known at this point).  The aim of this paper is to show that \emph{with little modification} many recursive counterexamples and reversals, second-order as they may be, also establish results in higher-order arithmetic (see Section \ref{prelim1} for the latter). 

\smallskip

As will become clear below, there is no `size restriction' on our results.  For instance, we first lift Specker \emph{sequences} as in item \eqref{gong} to Specker \emph{nets} indexed by Baire space in Section \ref{specknets}; the same proof then immediately provides a lifting to Specker nets \emph{indexed by sets of any cardinality} (expressible in the language at hand).  Intuitively, our results constitute some kind of empirical version of the L\"owenheim-Skolem theorem, which states that models of any cardinality exist.  

\smallskip

As to naming, \emph{Reverse Mathematics} derives its name from the `reverse' way of doing mathematics it embodies, namely to derive theorems from axioms rather than the other way around.  
This name was coined by Friedman (\cite{simpson2}*{p.\ 35}).
Perhaps the results in this paper should therefore be baptised \emph{upwards mathematics}, as it pushes theorems upwards in the (finite) type hierarchy.  

\smallskip

%

\smallskip

Finally, it is now an interesting question which other techniques of classical computability theory lift with similar ease into the higher-order world. 
Similarly, it is a natural question whether one can obtain a meta-theorem that encompasses all the below liftings.
We do not have an answer to these questions.  

\section{Preliminaries}\label{prelim}

We introduce \emph{Reverse Mathematics} in Section \ref{prelim1}, as well as its generalisation to \emph{higher-order arithmetic}, and the associated base theory $\RCAo$.  
We introduce some essential axioms in Section~\ref{prelim2}.

\subsection{Reverse Mathematics}\label{prelim1}
Reverse Mathematics is a program in the foundations of mathematics initiated around 1975 by Friedman (\cites{fried,fried2}) and developed extensively by Simpson (\cite{simpson2}).  
The aim of RM is to identify the minimal axioms needed to prove theorems of ordinary, i.e.\ non-set theoretical, mathematics.   In almost all cases, these minimal axioms are also \emph{equivalent} to the theorem at hand (over a weak logical system).  The reversal, i.e.\ the derivation of the minimal axioms from the theorem, is often proved based on a recursive counterexample to the latter, as noted in \cite{rauwziel}*{p.\ 1368}.  

\smallskip

We refer to \cite{stillebron} for a layman introduction to RM and to \cite{simpson2, simpson1} for an overview of RM.  We expect basic familiarity with RM, but do sketch some aspects of Kohlenbach's \emph{higher-order} RM (\cite{kohlenbach2}) essential to this paper, including the base theory $\RCAo$ (Definition \ref{kase}).  
As will become clear, the latter is officially a type theory but can accommodate (enough) set theory via Definition \ref{strijker} below. 

\smallskip

First of all, in contrast to `classical' RM based on \emph{second-order arithmetic} $\Z_{2}$, higher-order RM uses $\L_{\omega}$, the richer language of \emph{higher-order arithmetic}.  
Indeed, while $\L_{2}$, the language of $\Z_{2}$, is restricted to natural numbers and sets of natural numbers, higher-order arithmetic can accommodate sets of sets of natural numbers, sets of sets of sets of natural numbers, et cetera.  
To formalise this idea, we introduce the collection of \emph{all finite types} $\mathbf{T}$, defined by the two clauses:
\begin{center}
(i) $0\in \mathbf{T}$   and   (ii)  if $\sigma, \tau\in \mathbf{T}$ then $( \sigma \di \tau) \in \mathbf{T}$,
\end{center}
where $0$ is the type of natural numbers, and $\sigma\di \tau$ is the type of mappings from objects of type $\sigma$ to objects of type $\tau$.
In this way, $1\equiv 0\di 0$ is the type of functions from numbers to numbers, and where  $n+1\equiv n\di 0$.  Viewing sets as given by characteristic functions, we note that $\Z_{2}$ only includes objects of type $0$ and $1$.    

\smallskip

Secondly, the language $\L_{\omega}$ includes variables $x^{\rho}, y^{\rho}, z^{\rho},\dots$ of any finite type $\rho\in \mathbf{T}$.  Types may be omitted when they can be inferred from context.  
The constants of $\L_{\omega}$ include the type $0$ objects $0, 1$ and $ <_{0}, +_{0}, \times_{0},=_{0}$  which are intended to have their usual meaning as operations on $\N$.
For a type $\tau$ given as $\tau\equiv(\tau_{1}\di \dots\di \tau_{k}\di 0)$, the associated equality is defined in terms of `$=_{0}$' as follows: for any objects $x^{\tau}, y^{\tau}$, we have
\be\label{aparth}
[x=_{\tau}y] \equiv (\forall z_{1}^{\tau_{1}}\dots z_{k}^{\tau_{k}})[xz_{1}\dots z_{k}=_{0}yz_{1}\dots z_{k}],
\ee

Furthermore, $\L_{\omega}$ also includes the \emph{recursor constant} $\mathbf{R}_{0}$ which allows for iteration on type $0$-objects as as described below in \eqref{special}.  Formulas and terms are defined as usual.  
One obtains the sub-language $\L_{n+2}$ by restricting the above type formation rule to produce only type $n+1$ objects (and related types of similar complexity).        
\bdefi\label{kase} 
The base theory $\RCAo$ consists of the following axioms.
\begin{enumerate}
 \renewcommand{\theenumi}{\alph{enumi}}
\item  Basic axioms expressing that $0, 1, <_{0}, +_{0}, \times_{0}$ form an ordered semi-ring with equality $=_{0}$.
\item Basic axioms defining the well-known $\Pi$ and $\Sigma$ combinators (aka $K$ and $S$ in \cite{avi2}), which allow for the definition of \emph{$\lambda$-abstraction}. 
\item The defining axiom of the recursor constant $\mathbf{R}_{0}$: For $m^{0}$ and $f^{1}$: 
\be\label{special}
\mathbf{R}_{0}(f, m, 0):= m \textup{ and } \mathbf{R}_{0}(f, m, n+1):= f(n, \mathbf{R}_{0}(f, m, n)).
\ee
\item The \emph{axiom of extensionality}: for all $\rho, \tau\in \mathbf{T}$, we have:
\be\label{EXT}\tag{$\textsf{\textup{E}}_{\rho, \tau}$}  
(\forall  x^{\rho},y^{\rho}, \varphi^{\rho\di \tau}) \big[x=_{\rho} y \di \varphi(x)=_{\tau}\varphi(y)   \big].
\ee 
\item The induction axiom for quantifier-free\footnote{To be absolutely clear, variables (of any finite type) are allowed in quantifier-free formulas of the language $\L_{\omega}$: only quantifiers are banned.} formulas of $\L_{\omega}$.
\item $\QFAC^{1,0}$: The quantifier-free Axiom of Choice as in Definition \ref{QFAC}.
\end{enumerate}
\edefi
\bdefi\label{QFAC} The axiom $\QFAC$ consists of the following for all $\sigma, \tau \in \textbf{T}$:
\be\tag{$\QFAC^{\sigma,\tau}$}
(\forall x^{\sigma})(\exists y^{\tau})A(x, y)\di (\exists Y^{\sigma\di \tau})(\forall x^{\sigma})A(x, Y(x)),
\ee
for any quantifier-free formula $A$ in the language of $\L_{\omega}$.
\edefi
Now $\RCAo$ and $\RCA_{0}$ prove the same sentences `up to language' as the latter is set-based and the former function-based.  
In particular, there is a function-based formalisation of $\RCA_0$ called $\RCA_0^2$, while $\RCA_0^\omega$ is a
conservative extension of $\RCA_0^2$ (see \cite{kohlenbach2}*{\S2} for details).
Recursion as in equation \eqref{special} is called \emph{primitive recursion}.  

\smallskip

We use the usual notations for natural, rational, and real numbers, and the associated functions, as introduced in \cite{kohlenbach2}*{p.\ 288-289}.  
\begin{defi}[Real numbers and related notions in $\RCAo$]\label{keepintireal}\rm~
\begin{enumerate}
 \renewcommand{\theenumi}{\alph{enumi}}
\item Natural numbers correspond to type zero objects, and we use `$n^{0}$' and `$n\in \N$' interchangeably.  Rational numbers are defined as signed quotients of natural numbers, and `$q\in \Q$' and `$<_{\Q}$' have their usual meaning.    
\item Real numbers are coded by fast-converging Cauchy sequences $q_{(\cdot)}:\N\di \Q$, i.e.\  such that $(\forall n^{0}, i^{0})(|q_{n}-q_{n+i}|<_{\Q} \frac{1}{2^{n}})$.  
We use Kohlenbach's `hat function' from \cite{kohlenbach2}*{p.\ 289} to guarantee that every $q^{1}$ defines a real number.  
\item We write `$x\in \R$' to express that $x^{1}:=(q^{1}_{(\cdot)})$ represents a real as in the previous item and write $[x](k):=q_{k}$ for the $k$-th approximation of $x$.    
\item Two reals $x, y$ represented by $q_{(\cdot)}$ and $r_{(\cdot)}$ are \emph{equal}, denoted $x=_{\R}y$, if $(\forall n^{0})(|q_{n}-r_{n}|\leq {2^{-n+1}})$. Inequality `$<_{\R}$' is defined similarly.  
We sometimes omit the subscript `$\R$' if it is clear from context.           
\item Functions $F:\R\di \R$ are represented by $\Phi^{1\di 1}$ mapping equal reals to equal reals, i.e.\ satisfying $(\forall x , y\in \R)(x=_{\R}y\di \Phi(x)=_{\R}\Phi(y))$.\label{EXTEN}
\item The relation `$x\leq_{\tau}y$' is defined as in \eqref{aparth} but with `$\leq_{0}$' instead of `$=_{0}$'.  Binary sequences are denoted `$f^{1}, g^{1}\leq_{1}1$', but also `$f,g\in C$' or `$f, g\in 2^{\N}$'.  Elements of Baire space are given by $f^{1}, g^{1}$, but also denoted `$f, g\in \N^{\N}$'.
\item For a binary sequence $f^{1}$, the associated real in $[0,1]$ is $\r(f):=\sum_{n=0}^{\infty}\frac{f(n)}{2^{n+1}}$.\label{detrippe}
\item Sets of type $\rho$ objects $X^{\rho\di 0}, Y^{\rho\di 0}, \dots$ are given by their characteristic functions $F^{\rho\di 0}_{X}\leq_{\rho\di 0}1$, i.e.\ we write `$x\in X$' for $ F_{X}(x)=_{0}1$. \label{koer} 
\end{enumerate}
\end{defi}
\noindent
The following special case of item \eqref{koer} is singled out, as it will be used frequently.
\bdefi[$\RCAo$]\label{strijker}
A `subset $D$ of $\N^{\N}$' is given by its characteristic function $F_{D}^{2}\leq_{2}1$, i.e.\ we write `$f\in D$' for $ F_{D}(f)=1$ for any $f\in \N^{\N}$.
A `binary relation $\preceq$ on a subset $D$ of $\N^{\N}$' is given by some functional $G_{\preceq}^{(1\times 1)\di 0}$, namely we write `$f\preceq g$' for $G_{\preceq}(f, g)=1$ and any $f, g\in D$.
Assuming extensionality on the reals as in item \eqref{EXTEN}, we obtain characteristic functions that represent subsets of $\R$ and relations thereon.  
Following Definition \ref{keepintireal}, it is clear we can also represent sets of finite sequences of reals, and relations thereon.  
\edefi
As is clear from the previous, the formula `$x<_{\R}y$' for $x, y\in \R$ is a $\Sigma_{1}^{0}$-formula.   
To make sure this inequality defines a binary relation in the sense of Definition \ref{strijker}, we shall always assume $(\exists^{2})$ from Section \ref{prelim2}, as also discussed in Section \ref{netskes}
The crux is that $\exists^{2}$ makes (in)equality on the reals decidable. 

\smallskip

For completeness, we also list notational conventions for finite sequences.  
\begin{nota}[Finite sequences]\label{skim}\rm
We assume a dedicated type for `finite sequences of objects of type $\rho$', namely $\rho^{*}$.  Since the usual coding of pairs of numbers goes through in $\RCAo$, we shall not always distinguish between $0$ and $0^{*}$. 
Similarly, we do not always distinguish between `$s^{\rho}$' and `$\langle s^{\rho}\rangle$', where the former is `the object $s$ of type $\rho$', and the latter is `the sequence of type $\rho^{*}$ with only element $s^{\rho}$'.  The empty sequence for the type $\rho^{*}$ is denoted by `$\langle \rangle_{\rho}$', usually with the typing omitted.  

\smallskip

Furthermore, we denote by `$|s|=n$' the length of the finite sequence $s^{\rho^{*}}=\langle s_{0}^{\rho},s_{1}^{\rho},\dots,s_{n-1}^{\rho}\rangle$, where $|\langle\rangle|=0$, i.e.\ the empty sequence has length zero.
For sequences $s^{\rho^{*}}, t^{\rho^{*}}$, we denote by `$s*t$' the concatenation of $s$ and $t$, i.e.\ $(s*t)(i)=s(i)$ for $i<|s|$ and $(s*t)(j)=t(j-|s|)$ for $|s|\leq j< |s|+|t|$. For a sequence $s^{\rho^{*}}$, we define $\overline{s}N:=\langle s(0), s(1), \dots,  s(N-1)\rangle $ for $N^{0}<|s|$.  
For a sequence $\alpha^{0\di \rho}$, we also write $\overline{\alpha}N=\langle \alpha(0), \alpha(1),\dots, \alpha(N-1)\rangle$ for \emph{any} $N^{0}$.  By way of shorthand, 
$(\forall q^{\rho}\in Q^{\rho^{*}})A(q)$ abbreviates $(\forall i^{0}<|Q|)A(Q(i))$, which is (equivalent to) quantifier-free if $A$ is.   
\end{nota}

\subsection{Some axioms of higher-order RM}\label{prelim2}
We introduce some functionals which constitute the counterparts of some of the Big Five systems, in higher-order RM.
We use the formulation from \cite{kohlenbach2, dagsamIII}.  
First of all, $\ACA_{0}$ is readily derived from:
\begin{align}\label{mu}\tag{$\mu^{2}$}
(\exists \mu^{2})(\forall f^{1})\big[ (\exists n)(f(n)=0) \di [(f(\mu(f))=0)&\wedge (\forall i<\mu(f))f(i)\ne 0 ]\\
& \wedge [ (\forall n)(f(n)\ne0)\di   \mu(f)=0]    \big], \notag
\end{align}
and $\ACA_{0}^{\omega}\equiv\RCAo+(\mu^{2})$ proves the same sentences as $\ACA_{0}$ by \cite{hunterphd}*{Theorem~2.5}.   The (unique) functional $\mu^{2}$ in $(\mu^{2})$ is also called \emph{Feferman's $\mu$} (\cite{avi2}), 
and is clearly \emph{discontinuous} at $f=_{1}11\dots$; in fact, $(\mu^{2})$ is equivalent to the existence of $F:\R\di\R$ such that $F(x)=1$ if $x>_{\R}0$, and $0$ otherwise (\cite{kohlenbach2}*{\S3}), and to 
\be\label{muk}\tag{$\exists^{2}$}
(\exists \varphi^{2}\leq_{2}1)(\forall f^{1})\big[(\exists n)(f(n)=0) \asa \varphi(f)=0    \big]. 
\ee
\noindent
Secondly, $\FIVE$ is readily derived from the following sentence:
\be\tag{$\SS^{2}$}
(\exists\SS^{2}\leq_{2}1)(\forall f^{1})\big[  (\exists g^{1})(\forall n^{0})(f(\overline{g}n)=0)\asa \SS(f)=0  \big], 
\ee
and $\FIVE^{\omega}\equiv \RCAo+(\SS^{2})$ proves the same $\Pi_{3}^{1}$-sentences as $\FIVE$ by \cite{yamayamaharehare}*{Theorem 2.2}.   The (unique) functional $\SS^{2}$ in $(\SS^{2})$ is also called \emph{the Suslin functional} (\cite{kohlenbach2}).
By definition, the Suslin functional $\SS^{2}$ can decide whether a $\Sigma_{1}^{1}$-formula as in the left-hand side of $(\SS^{2})$ is true or false.   We analogously define the functional $\SS_{k}^{2}$ for $k>0$ which decides the truth or falsity of $\Sigma_{k}^{1}$-formulas in Kleene normal form (see \cite{simpson2}*{V.1.4}); we also define 
the system $\SIXK$ as $\RCAo+(\SS_{k}^{2})$, where  $(\SS_{k}^{2})$ expresses that $\SS_{k}^{2}$ exists.  Note that we allow formulas with \emph{function} parameters, but \textbf{not} \emph{functionals} here.
In fact, Gandy's \emph{Superjump} (\cite{supergandy}) constitutes a way of extending $\FIVE^{\omega}$ to parameters of type two.  
For completeness, the functional $\exists^{2}$ is also called $\SS_{0}^{2}$ and the system $\ACAo$ is also called $\Pi_{0}^{1}\textsf{-CA}^{\omega}$.  We note that the operators $\nu_{n}$ from \cite{boekskeopendoen}*{p.\ 129} are essentially $\SS_{n}^{2}$ strengthened to return a witness (if existant) to the $\Sigma_{n}^{1}$-formula at hand.  

\smallskip

\noindent
Thirdly, full second-order arithmetic $\Z_{2}$ is readily derived from $\cup_{k}\SIXK$, or from:
\be\tag{$\exists^{3}$}
(\exists E^{3}\leq_{3}1)(\forall Y^{2})\big[  (\exists f^{1})Y(f)=0\asa E(Y)=0  \big], 
\ee
and we therefore define $\Z_{2}^{\Omega}\equiv \RCAo+(\exists^{3})$ and $\Z_{2}^\omega\equiv \cup_{k}\SIXK$, which are conservative over $\Z_{2}$ by \cite{hunterphd}*{Cor.\ 2.6}. 
Despite this close connection, $\Z_{2}^{\omega}$ and $\Z_{2}^{\Omega}$ can behave quite differently, as discussed in e.g.\ \cite{dagsamIII}*{\S2.2}.   The functional from $(\exists^{3})$ is also called `$\exists^{3}$', and we use the same convention for other functionals. 

\smallskip

Next, the Heine-Borel theorem states the existence of a finite sub-cover for an open cover of certain spaces. 
Now, a functional $\Psi:\R\di \R^{+}$ gives rise to the \emph{canonical cover} $\cup_{x\in I} I_{x}^{\Psi}$ for $I\equiv [0,1]$, where $I_{x}^{\Psi}$ is the open interval $(x-\Psi(x), x+\Psi(x))$.  
Hence, the uncountable cover $\cup_{x\in I} I_{x}^{\Psi}$ has a finite sub-cover by the Heine-Borel theorem; in symbols:
\be\tag{$\HBU$}
(\forall \Psi:\R\di \R^{+})(\exists  y_{1}, \dots, y_{k}\in I){(\forall x\in I)}(\exists i\leq k)(x\in I_{y_{i}}^{\Psi}).
\ee
Note that $\HBU$ is \emph{Cousin's lemma} (see \cite{cousin1}*{p.\ 22}), i.e.\ the Heine-Borel theorem restricted to canonical covers.  
The latter restriction does not make much of a big difference, as studied in \cite{sahotop}.
By \cite{dagsamIII, dagsamV}, $\Z_{2}^{\Omega}$ proves $\HBU$ but $\Z_{2}^{\omega}+\QFAC^{0,1}$ cannot, 
and many basic properties of the \emph{gauge integral} (\cite{zwette, mullingitover}) are equivalent to $\HBU$.  

\smallskip

Finally, we list the following comprehension axiom, first introduced in \cite{samph}.
\bdefi[$\BOOT$]
$(\forall Y^{2})(\exists X^{1})(\forall n^{0})\big[ n\in X \asa (\exists f^{1})(Y(f, n)=0)    \big]. $
\edefi  
\noindent
Clearly, $\BOOT$ is inspired by $(\exists^{3})$ and $\Z_{2}^{\omega}+\QFAC^{0,1}$ cannot prove $\BOOT$ by the results in \cite{dagsamIII, dagsamV, samnetspilot}.
In this sense, $\BOOT$ is much harder to prove than $\ACA_{0}$, while $\RCAo+\BOOT$ is still a conservative extension of $\ACA_{0}$. 
Nonetheless, $\SIXK+\BOOT$ proves $\Pi_{k+1}^{1}$-$\CA_{0}$, as shown in \cite{samph}

\smallskip

We finish this section with some historical remarks pertaining to $\BOOT$.
\begin{rem}[Historical notes]\label{hist}\rm
First of all, the bootstrap principle $\BOOT$ is definable in Hilbert-Bernays' system $H$ from the \emph{Grundlagen der Mathematik} (see \cite{hillebilly2}*{Supplement IV}).  In particular, the functional $\nu$ from \cite{hillebilly2}*{p.\ 479} immediately\footnote{The functional $\nu$ from \cite{hillebilly2}*{p.\ 479} is such that if $(\exists f^{1})A(f)$, the function $(\nu f)A(f)$ is the lexicographically least such $f^{1}$.  The formula $A$ may contain type two parameters, as is clear from e.g.\ \cite{hillebilly2}*{p.\ 481} and other definitions.} yields the set $X$ from $\BOOT$ (and similarly for its generalisations), viewing the type two functional $Y^{2}$ as a parameter (which is done throughout \cite{hillebilly2}*{Supplement~IV}).  Thus, $\BOOT$ and $\ACA_{0}$ share the same historical roots.  

\smallskip

Secondly, Feferman's axiom \textsf{(Proj1)} from \cite{littlefef} is similar to $\BOOT$.  The former is however formulated using sets, which makes it more `explosive' than $\BOOT$ in that full $\Z_{2}$ follows when combined with $(\mu^{2})$, as noted in \cite{littlefef}*{I-12}.  The axiom \textsf{(Proj1)} only became known to us after the results in \cite{samph} were finished. 
\end{rem}

\section{Main results}\label{main}
We establish the results sketched in Section \ref{farg}.  In each section, we study a known recursive counterexample and show that it lifts to interesting results in higher-order arithmetic/about uncountable mathematics.  
\subsection{Specker nets}\label{specknets}
In Section \ref{specknetssub}, we lift the implication involving the monotone convergence theorem \emph{for sequences} ($\MCT_{\seq}^{[0,1]}$) and arithmetical comprehension ($\ACA_{0}$) to higher-order arithmetic.
This results in an implication involving the \emph{monotone convergence theorem} for \emph{nets} indexed by Baire space and the comprehension axiom $\BOOT$ from Section \ref{prelim2}.  
Nets and associated concepts are introduced in Section \ref{netskes}.  These results are based on \emph{Specker sequences}, discussed next.   

\smallskip

Indeed, the proof that the monotone convergence theorem implies $\ACA_{0}$ from \cite{simpson2}*{III.2} is based on a recursive 
counterexample by Specker (\cite{specker}), who proved the existence of a computable increasing sequence of rationals in 
the unit interval that does not converge to any computable real number.  We show that these results lift to the higher-order setting in that 
\emph{essentially the same proof} yields that the monotone convergence theorem \emph{for nets} indexed by Baire space ($\MCT_{\net}^{[0,1]}$) implies $\BOOT$.  
In particular, the notion of \emph{Specker sequence} readily generalises to \emph{Specker net}.
We provide full details for this case, going as far as comparing the original and `lifted' proof side-by-side.  A much less detailed proof was first published in \cite{samph}.

\subsubsection{Nets: basics and definitions}\label{netskes}
We introduce the notion of net and associated concepts.  
Intuitively speaking, nets are the generalisation of sequences to (possibly) uncountable index sets; nets are essential for convergence in topology and domain theory.  
On a historical note, Moore-Smith and Vietoris independently introduced these notions about a century ago in \cites{moorsmidje, kliet}, which is why nets are also called \emph{Moore-Smith sequences}.
Nets and filters yield the same convergence theory, but e.g.\ third-order nets are represented by fourth-order filters, i.e.\ nets are more economical in terms of type complexity (see \cite{zonderfilter}).

\smallskip

We use the following standard definition from \cite{ooskelly}*{Ch.\ 2}.
\bdefi[Nets]\label{nets}
A set $D\ne \emptyset$ with a binary relation `$\preceq$' is \emph{directed} if
\begin{enumerate}
 \renewcommand{\theenumi}{\alph{enumi}}
\item The relation $\preceq$ is transitive, i.e.\ $(\forall x, y, z\in D)([x\preceq y\wedge y\preceq z] \di x\preceq z )$.
\item For $x, y \in D$, there is $z\in D$ such that $x\preceq z\wedge y\preceq z$.\label{bulk}
\item The relation $\preceq$ is reflexive, i.e.\ $(\forall x\in D)(x \preceq x)$.  
\end{enumerate}
For such $(D, \preceq)$ and topological space $X$, any mapping $x:D\di X$ is a \emph{net} in $X$.  
We denote $\lambda d. x(d)$ as `$x_{d}$' or `$x_{d}:D\di X$' to suggest the connection to sequences.  
The directed set $(D, \preceq)$ is not always explicitly mentioned together with a net $x_{d}$.
\edefi 
We only use directed sets that are subsets of $\N^{\N}$, i.e.\ as given by Definition \ref{strijker}.   
In other words, we only study nets $x_{d}:D\di \R$ where $D$ is a subset of $\N^{\N}$, but any binary relation is allowed.  
Thus, a net $x_{d}$ in $\R$ is just a type $1\di 1$ functional with extra 
structure on its domain $D$ provided by `$\preceq$' as in Definition~\ref{strijker}, i.e.\ part of \emph{third-order} arithmetic. 


\smallskip

The definitions of convergence and increasing net are of course familiar.  
\bdefi[Convergence of nets]\label{convnet}
If $x_{d}$ is a net in $X$, we say that $x_{d}$ \emph{converges} to the limit $\lim_{d} x_{d}=y\in X$ if for every neighbourhood $U$ of $y$, there is $d_{0}\in D$ such that for all $e\succeq d_{0}$, $x_{e}\in U$. 
\edefi
\bdefi[Increasing nets]\label{crco}
A net $x_{d}:D\di \R$ is \emph{increasing} if $a\preceq b$ implies $x_{a}\leq_{\R} x_{b} $ for all $a,b\in D$.
\edefi
Many (convergence) notions concerning sequences carry over to nets \emph{mutatis mutandis}.  A rather general RM study of nets may be found in \cites{samcie19, samwollic19, samph,samnetspilot}.
We shall study the monotone convergence theorem for nets as follows.  
\bdefi[$\MCT_{\net}^{[0,1]}$]
Any increasing net in $[0,1]$ indexed by $\N^{\N}$ converges. 
\edefi
As discussed just below Definition \ref{strijker}, we shall study $\MCT_{\net}^{[0,1]}$ assuming $\ACAo$ to guarantee 
that `$\leq_{\R}$' is a binary relation as in the former definition.  This is a mere technicality that does however simplify proofs.  

\smallskip

The `original' monotone convergence theorem for \emph{sequences} as in \cite{simpson2}*{III.2} is denoted $\MCT_{\seq}^{[0,1]}$.
The implications $\MCT_{\seq}^{[0,1]}\leftarrow \ACA_{0}$ and $\MCT_{\net}^{[0,1]}\leftarrow \BOOT$ are in fact proved in exactly the same way.   

\smallskip

Finally, sequences are nets with index set $(\N, \leq_{\N})$ and theorems pertaining to nets therefore apply to sequences.  
However, some care is advised as e.g.\ a sub-net of a sequence is not necessarily a \emph{sub-sequence} (see \cite{samnetspilot}*{\S3}).

\subsubsection{Specker nets and comprehension}\label{specknetssub}
In this section, we show that $\MCT_{\net}^{[0,1]}\di \BOOT$ using a minor variation of the well-known proof $\MCT_{\seq}^{[0,1]}\di \ACA_{0}$ from \cite{simpson2}*{III.2.2} involving Specker sequences.

\smallskip

First of all, we distill the essence of the latter proof, as follows.  
\begin{enumerate}
\renewcommand{\theenumi}{\roman{enumi}}
\item We prove $\MCT_{\seq}^{[0,1]}\di \range$, where the latter states that the range exists for injective\footnote{Note that $f^{1}$ is `injective' or `one-to-one' if $(\forall n^{0}, m^{0})(f(n)=f(m)\di n=m)$.   This condition guarantees that the Specker sequence $(c_{n})_{n\in \N}$ from item \eqref{speck} does not escape $[0,2]$.  For general (non-injective) $g^{1}$, one can define $d_{n}:= \sum_{i=0}^{n}\frac{K(i)}{2^{g(i)}}$ where $K(i)$ is $1$ if $(\forall j<i)(g(j)\ne g(i))$, and $0$ otherwise.  
The proof in items \eqref{firstlaste}-\eqref{speck3} goes through for $(c_{n})_{n\in \N}$ replaced by $(d_{n})_{n\in \N}$, while the proof of Theorem \ref{proofofconcept} can be modified similarly.\label{wiref}} functions $f^{1}$, i.e.\ $(\exists X\subset \N)(\forall k^{0})\big(k \in X\asa (\exists m^{0})(f(m)=k)\big)$.\label{firstlaste}
\item Fix injective $f^{1}$ and define the Specker sequence $c_{n}:=\sum_{i=0}^{n}2^{-f(i)}$.\label{speck}
\item Note that $\MCT_{\seq}^{[0,1]}$ applies and let $c$ be $\lim_{n\di \infty}c_{n}$.  \label{speck2}
\item Establish the following equivalence:\label{speck3}
\be\label{kikoporg}
(\exists m^{0})(f(m)=k)\asa (\forall n^{0})\big( |c_{n}-c|<2^{-k}\di (\exists i\leq n)(f(i)=k)     \big).
\ee
\item Apply $\Delta_{1}^{0}$-comprehension to \eqref{kikoporg}, yielding the set $X$ needed for $\range$.\label{korifee}
\end{enumerate}
We now show how to lift the previous steps to higher-order arithmetic, resulting in a proof of $\MCT_{\net}^{[0,1]}\di \BOOT$ in Theorem \ref{proofofconcept}.

\smallskip

Regarding item \eqref{korifee}, to lift proofs involving $\Delta_{1}^{0}$-comprehension to the higher-order framework, we introduce the following comprehension axiom: 
\begin{align}
(\forall Y^{2}, Z^{2})\big[ (\forall n^{0})\big( (\exists f^{1})&(Y(f, n)=0) \asa (\forall g^{1})(Z(g, n)=0) \big)\tag{$\Delta$-$\CA$} \\
& \di (\exists X^{1})(\forall n^{0})(n\in X\asa (\exists f^{1})(Y(f, n)=0)\big].\notag
\end{align}
A snippet of countable choice suffices to prove $\Delta$-comprehension, as follows.
\begin{thm}\label{DELTA}
The system $\RCAo+\QFAC^{0,1}$ proves $\Delta\text{-}\CA$.
\end{thm}
\begin{proof}
The antecedent of $\Delta$-comprehension implies the following
\be\label{hani}
(\forall n^{0})(\exists g^{1}, f^{1})( Z(g, n)=0\di Y(f, n)=0 ).
\ee
Applying $\QFAC^{0,1}$ to \eqref{hani} yields $\Phi^{0\di 1}$ such that 
\be\label{fok}
(\forall n^{0})\big( (\forall g^{1})(Z(g, n)=0)\di Y(\Phi(n), n)=0 \big),
\ee
and by assumption an equivalence holds in \eqref{fok}, and we are done. 
\end{proof}
\noindent
The previous is not \emph{spielerei}: the crux of numerous reversals $T\di \ACA_{0}$ is that the theorem $T$ somehow allows for the reduction of $\Sigma_{1}^{0}$-formulas to $\Delta_{1}^{0}$-formulas, 
while $\Delta_{1}^{0}$-comprehension -included in $\RCA_{0}$- then yields the required $\Sigma_{1}^{0}$-comprehension, and $\ACA_{0}$ follows.  
Additional motivation for $\Delta\text{-}\CA$ is provided by Theorem \ref{motich}, 
while interesting RM results for $\Delta$-$\textsf{CA}$ can be found in \cite{dagsamX, dagsamIX}.

\smallskip

Regarding item \eqref{firstlaste}, lifting $\range$ to the higher-order framework is fairly basic: we just consider the existence of the range of \emph{type two} functionals (rather than type one functions), as in $\RANGE$ below.     
\begin{thm}\label{rage}
The system $\RCAo$ proves that $\BOOT$ is equivalent to 
\be\label{myhunt}\tag{$\RANGE$}
(\forall G^{2})(\exists X^{1})(\forall n^{0})\big[n\in X\asa (\exists f^{1})(G(f)=n)  ].
\ee
\end{thm}
\begin{proof}
The forward direction is immediate.  For the reverse direction, define $G^{2}$ as follows for $n^{0}$ and $g^{1}$: put $G(\langle n\rangle *g)=n+1$ if $Y(g, n)=0$, and $0$ otherwise. 
Let $X\subseteq \N$ be as in $\RANGE$ and note that 
\be\label{taff}
(\forall m^{0}\geq 1 )( m\in X \asa (\exists f^{1})(G(f)=m)\asa (\exists g^{1})(Y(g, m-1)=0)  ).
\ee
which is as required for $\BOOT$ after trivial modification. 
\end{proof}
This theorem was first proved as \cite{samph}*{Theorem 3.19}.
Again, the previous is not a gimmick: reversals involving $\ACA_{0}$ are often established using $\range$, and those yield implications involving $\RANGE$, for instance as follows. 
\begin{thm}\label{proofofconcept}
The system $\ACAo+\Delta\text{-}\CA$ proves $\MCT_{\net}^{[0,1]}\di \BOOT$.
\end{thm}
\begin{proof}
%
We shall establish $\RANGE$ and obtain $\BOOT$ by Theorem \ref{rage}, which mimics the above item \eqref{firstlaste}. 
Fix $Y:\N^{\N}\di \N$
and let $(D, \preceq_{D})$ be a directed set with $D$ consisting of the finite sequences $w^{1^{*}}$ in $\N^{\N}$ such that 
\be\label{kat}
(\forall i, j<|w|)(Y(w(i))=_{0}Y(w(j))\di i=_{0}j).
\ee
Define $v\preceq_{D} w $ if $ (\forall i<|v|)(\exists j<|w|)(v(i)=_{1}w(j))$ for any $v^{1^{*}}, w^{1^{*}}$. 
Note that $(\exists^{2})$ is necessary for this definition.  
Following item \eqref{speck}, fix some $Y^{2}$ and define the `Specker net' $c_{w}:D\di [0,1]$ as $c_{w}:= \sum_{i=0}^{|w|-1}2^{-Y(w(i))}$.  
Note that this net is in $[0,2]$ thanks to \eqref{kat}.  
Clearly, $c_{w}$ is increasing as in Definition \ref{crco} and let $c$ be the limit provided by $\MCT_{\net}^{[0,1]}$, following item \eqref{speck2}.  
Following item \eqref{speck3}, consider the following generalisation of \eqref{kikoporg}, for any $k\in \N$:
\be\label{kikop}
(\exists f^{1})(Y(f)=k)\asa (\forall w^{1^{*}})\big( |c_{w}-c|<2^{-k}\di (\exists g\in w)(Y(g)=k)     \big), 
\ee
for which the reverse direction is trivial thanks to $\lim_{w}c_{w}=c$.  
For the forward direction in \eqref{kikop}, assume the left-hand side holds for $f=f_{1}^{1}$ and fix some $w_{0}^{1^{*}}$ such that $|c-c_{w_{0}}|<\frac{1}{2^{k}}$.  
Since $c_{w}$ is increasing, we also have $|c-c_{w}|<\frac{1}{2^{k}}$ for $w\succeq_{D} w_{0}$.  
Now there must be $f_{0}$ in $w_{0}$ such that $Y(f_{0})=k$, as otherwise $w_{1}=w_{0}*\langle f_{1}\rangle$ satisfies $ w_{1}\succeq_{D}w_{0}$ but also $c_{w_{1}}>c$, which is impossible.  

\smallskip

Note that \eqref{kikop} has the right form to apply $\Delta\text{-}\CA$ (modulo obvious coding), and the latter provides the set required by $\RANGE$, following item \eqref{korifee}.
\end{proof}
We refer to the net $c_{w}$ from the proof as a \emph{Specker net} following the concept of \emph{Specker sequence} pioneered in \cite{specker}.
We hope that the reader agrees that the previous proof is \emph{exactly} the final part of the proof of \cite{simpson2}*{III.2.2} as in items \eqref{firstlaste}-\eqref{korifee}, save for the replacement of sequences by nets and functions by functionals. 
The aforementioned `reuse' comes at a cost however: the proof of $\MCT_{\net}^{[0,1]}\asa \BOOT$ in \cite{samph}*{\S3.2} does not make use of countable choice or $\Delta\text{-}\CA$.  
Moreover, from the proof of this equivalence, once can essentially `read off' that a realiser for $\MCT_{\net}^{[0,1]}$ computes $\exists^{3}$ in the sense of Kleene's S1-S9, and vice versa (see also \cite{samnetspilot}*{\S3.1}).   
It seems one cannot obtain this S1-S9 result from the above proof because of $\Delta\text{-}\CA$.

\smallskip

Now, Theorem \ref{proofofconcept} readily generalises by increasing the size of the index sets to any set of objects of finite type.  
The case of nets indexed by $\mathcal{N}\equiv\N^{\N}\di \N$ may be found in \cite{samph}*{Theorem 3.38}.
In particular, the monotone convergence theorem for nets indexed by $\mathcal{N}$ in $[0,1]$ implies the following comprehension axiom:
\be\label{myhunt1}\tag{$\RANGE^{1}$}
(\forall G^{3})(\exists X^{1})(\forall n^{0})\big[n\in X\asa (\exists Y^{2})(G(Y)=n)  ], 
\ee
which states the existence of the range of type three functionals.  It goes without saying that the proof of Theorem \ref{proofofconcept} readily\footnote{As is clear from \eqref{kikop}, the proof of Theorem \ref{proofofconcept} makes use of finite sequences $w^{1^{*}}$, which are readily coded as type $1$ objects in $\RCAo$.  The generalisation of the proof of Theorem \ref{proofofconcept} to $\mathcal{N}$ and $\RANGE^{1}$ thus uses finite sequences $w^{2^{*}}$.  Hence, one needs a version of $\Delta$-comprehension for the latter variables, which in turn follows from $\QFAC^{0,2^{*}}$, following Notation \ref{skim}.} lifts 
to nets indexed by $\mathcal{N}$. 
\bdefi[$\MCT_{\net}^{\mathcal{N}}$]
Any increasing net in $[0,1]$ indexed by $\mathcal{N}$ converges. 
\edefi
Following Theorems \ref{DELTA} and \ref{proofofconcept}, we have the following corollary.
\begin{cor}\label{corcor}
The system $\Z_{2}^{\Omega}+\QFAC^{0,2^{*}}$ proves $\MCT_{\net}^{ \mathcal{N}}\di \RANGE^{1}$.
\end{cor}
One could weaken the base theory in the previous corollary by modifying the definition of net:  
the latter requires the binary relation `$\preceq$' to be decidable, for which $\exists^{2}$ (resp.\ $\exists^{3}$) is needed in Theorem \ref{proofofconcept} (resp.\ Corollary \ref{corcor}).
This is however the only spot where these functionals are needed. 

\smallskip

Finally, as noted in Footnote \ref{wiref}, while $\range$ is restricted to injective functions, one readily adapts the proof of \cite{simpson2}*{III.2.2} to work for general (non-injective) functions.  
The same holds for many of the proofs discussed below, and we will not always point this out explicitly.  Since $\range\asa \ACA_{0}$, it should be no surprise that 
this is possible.  The same remark hold for $\WKL$, $\Sigma_{1}^{0}$-separation, and the separation of the disjoint ranges of \emph{injective} functions, all equivalent by \cite{simpson2}*{IV.4.4}.  

\subsection{Compactness of metric spaces}\label{metricsect}
Complete separable metric spaces are represented (or: coded) in second-order arithmetic by countable dense subsets with a pseudo-metric (see e.g.\ \cite{simpson2, browner}).  
Various notions of compactness can then be formulated and their relations have been analysed in detail (see e.g.\ \cite{browner}).
In this section, we lift some of these results to higher-order arithmetic.

\smallskip

Our starting point is \cite{browner}*{Theorem 3.13}, which establishes the equivalence between $\ACA_{0}$ and \emph{a \(countable\) Heine-Borel compact complete metric space is totally bounded}.
The reverse implication is established via $\range$ and we shall lift this result to higher-order arithmetic.  
We shall make use of the standard definition of metric spaces, which does not use coding and can be found verbatim in e.g.\ \cite{rudin3,royden}.
Similar to Section \ref{netskes}, we assume $\ACAo$ to guarantee that `$=_{\R}$' yields a binary relation as in Definition \ref{strijker}.
\bdefi\label{donc}
A \emph{complete metric space $\tilde{D}$ over $\N^{\N}$} consists of $D\subseteq \N^{\N}$, an equivalence\footnote{An equivalence relation is a binary relation that is reflexive, transitive, and symmetric.} relation $=_{D}$, and $d: (D\times D)\di \R$ such that for all $ e, f, g \in D$:
\begin{enumerate}
 \renewcommand{\theenumi}{\alph{enumi}}
\item $d(e, f)=_{\R}0 \asa  e=_{D}f$,
\item $0\leq_{\R} d(e, f)=_{\R}d(f, e), $
\item $d(f, e)\leq_{\R} d(f, g)+ d(g, e)$, 
\end{enumerate}
and such that every Cauchy sequence in ${D}$ converges to some element in $D$.
\edefi
To be absolutely clear, the final condition regarding $\tilde{D}$ in the definition means that if $\lambda n.f_{n}$ is a sequence in $D$ such that $(\forall k^{0})(\exists N^{})(\forall  m^{0}, n^{0}\geq N)(d(f_{n}, f_{m})<_{\R}\frac{1}{2^{k}})$, then there is $g\in D$ such that 
$(\forall k^{0})(\exists n^{0})(\forall m\geq n)(d(f_{m}, g)<\frac{1}{2^{k}})$.
A \emph{point} in $\tilde{D}$ is just any element in $D$.  
Two points $e, f\in \tilde{D}$ are said to be \emph{equal} if $e=_{D}f$.  Note that the `hat function' from \cite{kohlenbach2} readily yields $\R$ as a metric space over $\N^{\N}$.

\smallskip

We shall use standard notation like $B(e, r)$ for the open ball $\{f\in D: d(f, e)<_{\R}r\}$.  The first item in Definition \eqref{donc} expresses a kind of extensionality property and we tacitly assume that 
every mapping with domain $D$ respects `$=_{D}$'.  
\bdefi[Heine-Borel]\label{koeka}
A complete metric space $\tilde{D}$ over $\N^{\N}$ is \emph{Heine-Borel compact} if for any $Y: {D}\di \R^{+}$, the cover $\cup_{e\in {D}} B(e, Y(e))$ has a finite sub-cover. 
\edefi
We define \emph{countable} Heine-Borel compactness as the previous definition restricted to \emph{countable} covers of $D$, i.e.\ there is a sequence $(d_{n})_{n\in \N}$ in $D$ and a sequence of rationals $(r_{n})_{n\in \N}$ such that $D\subset \cup_{n\in \N}B(d_{n}, r_{n})$. 
\bdefi[Totally bounded]\label{totallydude}
A complete metric space $\tilde{D}$ over $\N^{\N}$ is \emph{totally bounded} if there is a sequence of finite sequences $\lambda n.x_{n}^{0\di 1^{^{*}}}$ of points in $\tilde{D}$ such that for any $x\in {D}$ there is $n\in \N$ such that $d(x, x_{n}(i))<2^{-n}$ for some $i<|x_{n}|$.
\edefi
We now obtain the following theorem by lifting the proof\footnote{The proof of \cite{browner}*{Theorem 3.13} shows that $\range$ follows from item \eqref{honk} in Theorem \ref{cokkl} when formulated in $\L_{2}$.  
However, this proof goes through \emph{verbatim} for general functions, i.e.\ the restriction to injections (as in $\range$) is not necessary.} of \cite{browner}*{Theorem 3.13}.
We let $\IND$ be the induction axiom for all formulas in $\L_{\omega}$.  

\smallskip
\begin{thm}\label{cokkl}
$\ACAo+\IND$ proves that each of the following items:
\begin{enumerate}
 \renewcommand{\theenumi}{\alph{enumi}}
\item {a Heine-Borel compact complete metric space over $\N^{\N}$ is totally bounded},\label{honk}
\item item \eqref{honk} restricted to countable Heine-Borel compactness,\label{kanwel}
\item item \eqref{honk} with sequential compactness instead of Heine-Borel compactness,\label{kanwel2}
\end{enumerate}
 implies the comprehension axiom $\BOOT$.
\end{thm}
\begin{proof}
We derive $\RANGE$ from item \eqref{honk}, and $\BOOT$ is therefore immediate; the implication involving item \eqref{kanwel} is then immediate.  
Fix some $Y^{2}$ and define $D=\N^{\N}\cup \{0_{D}\}$ where $0_{D}$ is some special element.
Define $f=_{D}e$ as $Y(f)=_{0}Y(e)$ for any $e, f\in D\setminus\{0_{D}\}$, while $0_{D}=0_{D}$ is defined as true and $f=_{D}0_{D}$ is defined as false for $d\in D\setminus\{0_{D} \}$.  

\smallskip

Define $d:D^{2}\di \R$ as follows: $d(f, g)= |2^{-Y(f)}-2^{-Y(g)}|$ if $f, g\ne0_{D} $, $d(0_{D}, 0_{D})=0_{D}$, and $d(0_{D}, f)=d(f, 0_{D})= 2^{-Y(f)}$ for $f \ne 0_{D}$.   
Clearly, this is a metric space in the sense of Definition~\ref{donc} and the `zero element' $0_{D}$ satisfies $\lim_{n\di \infty}d(0_{D},f_{n} )=_{\R}0$, assuming $Y$ is unbounded on $\N^{\N}$ and $\lambda n.f_{n}$ is a sequence in $D$ witnessing this, i.e.\ $Y(f_{n})\geq n$ for any $n\in \N$.

\smallskip

Now, given a Cauchy sequence $\lambda n.f_{n}$ in $D$, either it converges to $0_{D}$ or $d(0_{D}, f_{n})$ is eventually constant, i.e.\ the completeness property of $\tilde{D}$ is satisfied.  
Moreover, the Heine-Borel property as in Definition \ref{koeka} is also straightforward, as any neighbourhood of $0_{D}$ covers all but finitely many $2^{-Y(f)}$ for $f\in \N^{\N}$ by definition.
One seems to need $\IND$ to form the finite sub-cover.  
Let $\lambda n.x_{n}$ be the sequence provided by item \eqref{honk} that witnesses that $\tilde{D}$ is totally bounded.   
Now define $X\subseteq \N$ as:
\be\label{furk}
 n\in X \asa (\exists i<|x_{n+1}|)\big[2^{-n}=_{\R}d(0_{D}, x_{n+1}(i))], 
\ee
and one readily shows that $n\in X\asa (\exists f^{1})(Y(f)=n)$, i.e.\ $\RANGE$ follows. 
Note that one can remove `$=_{\R}$' from \eqref{furk} in favour of a decidable equality. 

\smallskip

Regarding item \eqref{kanwel2}, if a sequence $\lambda n.f_{n}$ is `unbounded' as in $(\forall m^{0})(\exists n^{0})(Y(f_{n})> m)$, then there is an obvious sub-sequence that converges to $0_{D}$.  
In case we have $(\exists m^{0})(\forall n^{0})(Y(f_{n})\leq m)$, there is a constant sub-sequence, and the space $\tilde{D}$ is clearly sequentially compact. 
\end{proof}
%
Now, Definition \ref{totallydude} is used in RM (see e.g.\ \cites{simpson2, browner}) and is sometimes referred to as \emph{effective} total boundedness as there
is a sequence that enumerates the finite sequences of approximating points.  As it turns out, this extra information yields countable choice in the higher-order setting. 
Note that the monotone convergence theorem for nets \emph{with a modulus of convergence} similarly yields $\BOOT+\QFAC^{0,1}$ by \cite{samph}*{\S3.3}; obtaining countable choice in this context therefore seems normal.
\begin{cor}\label{rampant}
The system $\ACAo+\IND$ proves $[\text{item } \eqref{honk} \di \QFAC^{0,1}]$.
\end{cor}
\begin{proof}
In light of $n\in X\asa (\exists f^{1})(Y(f)=n)$ and \eqref{furk}, one of the $x_{n+1}(i)$ for $i<|x_{n+1}|$ provides a witness to $(\exists f^{1})(Y(f)=n)$, if such there is.  
\end{proof}
Finally, one can generalise the previous to higher types.  For instance, Definition~\ref{donc} obviously generalises \emph{mutatis mutandis} to yield the definition of {complete metric spaces $\tilde{D}$ {over $\mathcal{N}\equiv\N^{\N}\di \N$}}, and the same for any finite type.  As opposed to nets indexed by function spaces like $\mathcal{N}$, a metric space based on the latter is quite standard. 
The proof of Theorem \ref{cokkl} and Corollary \ref{rampant} can then be relativised.  
\begin{cor}\label{proofofconcepts}
$\Z^{\Omega}_{2}$ proves that the following:
\begin{enumerate}
 \renewcommand{\theenumi}{\alph{enumi}}
 \setcounter{enumi}{3}
\item {a countable Heine-Borel compact complete metric space over $\mathcal{N}$ is totally bounded},\label{honktoet}
\end{enumerate}
implies the comprehension axiom $\RANGE^{1}$.
\end{cor}
\begin{cor}\label{rampantq}
The system $\Z_{2}^{\Omega}$ proves the implication $[\text{item } \eqref{honktoet} \di \QFAC^{0,2}]$.
\end{cor}
Note that $\RANGE^{1}$ was first introduced in \cite{samph}*{\S3.7} and follows from the monotone convergence theorem for nets \emph{indexed by} $\mathcal{N}$.
In fact, the usual proof of the monotone convergence theorem involving Specker sequences immediately generalises to Specker nets indexed by $\mathcal{N}$, as discussed in Section \ref{specknetssub}.
Similarly, item~\eqref{honk} from Theorem \ref{cokkl} for `larger' spaces yields $\QFAC^{0,n}$ for larger $n$ and $\RANGE$ generalised to higher-order functionals with `larger' domain.  

%

\subsection{Closed sets in metric spaces}\label{clookes}
The notion of `closed set' can be defined in various (equivalent) ways, e.g.\ as the complement of an open set, or as the union of a set together with its limit points.  
The connection between these different definitions in RM is studied in \cite{browner2}*{\S2} and we shall lift one of these results to higher-order arithmetic, namely 
that \emph{a separably RM-closed set in $[0,1]$ is also RM-closed} implies $\range$, as proved in \cite{browner2}*{Theorem 2.9}.

\smallskip

First of all, a \emph{separably RM-closed} set $\overline{S}$ in a metric space is given in RM by a sequence $\lambda n.x_{n}$ and `$x\in \overline{S}$' is then $(\forall k^{0})(\exists n^{0})(d(x,x_{n})<\frac{1}{2^{k}})$, 
where $d$ is the metric of the space.  Intuitively, the set is represented by a countable dense subset.   
An \emph{RM-closed set} is given by the complement of an RM-open set, and the latter is given by a countable union of open intervals.  
In this way, $\Sigma_{1}^{0}$ (resp.\ $\Pi_{1}^{0}$) formulas exactly correspond to open (resp.\ closed) sets (see \cite{simpson2}*{II.5.7}).  

\smallskip

Secondly, it is well-known that $\ZF$ cannot prove that `$\R$ is a sequential space', i.e.\ the equivalence between the definition of closed and sequentially closed set; countable choice however suffices (see \cite{heerlijkheid}*{p.\ 73}).   
On the other hand, we can avoid the Axiom of Choice by replacing sequences with nets everywhere, as shown in \cite{samnetspilot}*{\S4.4}.  
In this light, the following definition and theorem make sense.   
\bdefi
A separably closed set $\overline{S}$ in $\R$ is given by a net $x_{d}:D\di \Q$ with $D\subseteq\N^{\N}$ and where $x\in \overline{S}$ is given by  $(\forall k^{0})(\exists d\in D)(|x-x_{d}|\leq \frac{1}{2^{k}})$.
\edefi
\begin{thm}[$\CLO$]
A separably closed set in $\R$ is RM-closed.  
\end{thm}
The following is obtained by lifting the proof\footnote{The proof of \cite{browner2}*{Theorem 2.9} shows that $\range$ follows from $\CLO$ for separably RM-closed sets.  
However, this proof goes through \emph{verbatim} for general functions, i.e.\ the restriction to injections (as in $\range$) is not necessary.} of \cite{browner2}*{Theorem 2.9}.  
As in Section~\ref{netskes}, we assume $(\exists^{2})$ to guarantee `$\leq_{\R}$' is a binary relation as in Definition~\ref{strijker}. 
\begin{thm}\label{dontlabel}
The system $\ACAo+\Delta\text{-}\CA$ proves $\CLO\di \RANGE$.
\end{thm}
\begin{proof}
Fix some $Y^{2}$, define $D=(\N^{\N}\times \N)$ and let $\overline{S}$ be given by the net $x_{d}$ defined as $ 2^{-Y(f)}$ if $d=(f, 0)$, and zero otherwise. 
The lexicographic order on $\N^{\N}$ yields a directed set (using $\exists^{2}$).  By $\CLO$, $\overline{S}$ is also RM-closed, i.e.\ there is some $\Pi_{1}^{0}$-formula $\varphi(x)$ such that $x\in \overline{S}\asa \varphi(x)$.  
Hence, $\Delta\text{-}\CA$ (modulo $\exists^{2}$) applies to:
\be\label{equikes}
2^{-n}\in \overline{S}\asa  (\exists f^{1})(Y(f)=n),
\ee
yielding $\RANGE$.  To prove \eqref{equikes}, the reverse direction is immediate.  For the forward direction, if $2^{-n}\in \overline{S}$, then by definition $(\exists d_{0}\in D)(|2^{-n}-x_{d_{0}}|\leq \frac{1}{2^{k_{0}}})$, 
where $k_{0}$ is large enough to guarantee that $\frac{1}{2^{n+1}}< \frac{1}{2^{n}}-\frac{1}{2^{k_{0}}}$.  For such $d_{0}\in D$, we must have $d_{0}=(f_{0}, 0)$ and $Y(f_{0})=n$ by the definition of $x_{d}$, and \eqref{equikes} follows.
\end{proof}
Note that we can use any notion of closed set in the consequent of $\CLO$, as long as $\Delta\text{-}\CA$ still applies to \eqref{equikes}.  
In fact, we study a (more) general notion of open set in \cite{samph}*{\S4} based on the right-hand side of $\BOOT$.  
Finally, a set $C$ is `located' if $d(x, C)=\inf_{y\in C}d(x, y)$ exists as a continuous function.  
\begin{cor}
The theorem holds if we replace in $\CLO$ `RM-closed' by `located'.
\end{cor}
\begin{proof}
The set $\overline{S}$ from the proof of the theorem is now located; \eqref{equikes} then becomes:
\[
d(2^{-n}, \overline{S})=_{\R}0\asa  (\exists f^{1})(Y(f)=n),
\]
Thus, $\BOOT$ follows in the same way as in the proof of the theorem.  
\end{proof}
Finally, it goes without saying that one can obtain $\RANGE^{1}$ from $\CLO$ generalised to sub-sets of $\mathcal{N}$, while
`even larger' sets give rise to the existence of the range of functionals with `even larger' domain. 

\subsection{Rado selection lemma}\label{radokes}
We study the \emph{Rado selection lemma} from \cite{radiant}.  
The countable version of this lemma is equivalent to $\ACA_{0}$ by \cite{simpson2}*{III.7.8}, while a proof based on $\range$ can be found in \cite{hirstphd}*{\S3}.
We shall lift the reversal to higher-order arithmetic, making use of $\RANGE$.
We first need some definitions. 
\bdefi
A \emph{choice function} $f$ for a collection of non-empty $A_{i}$ indexed by $I$, is such that $f(i)\in A_{i}$ for all $i\in I$.
\edefi
A collection of finite subsets of $\N$ indexed by $\N^{\N}$ is of course given by a mapping $Y^{1\di 0^{*}}$.
In case the latter is continuous, the index set is actually countable.  

\bdefi[$\rado(\N^{\N})$]
Let $A_{i}$ be a collection of finite sets indexed by $\N^{\N}$ and let $F_{J}^{2}$ be a choice function for the collection $A_{j}$ for $j\in J$, for any finite set $J\subset \N^{\N}$.  
Then there is a choice function $F^{2}$ for the entire collection $A_{j}$ such that for all finite $J\subset \N^{\N}$, there is a finite $K\supseteq J$ such that for $j\in J$, $F(j)=_{0}F_{K}(j)$. 
\edefi
The following theorem is obtained by lifting the proof of \cite{hirstphd}*{Theorem 3.30} to higher-order arithmetic. 
\begin{thm}\label{dog}
The system $\RCAo$ proves $\rado(\N^{\N})\di \BOOT$.
\end{thm}
\begin{proof}
We will prove $\RANGE$, i.e.\ fix some $G^{2}$.
For any $w^{1^{*}}$, define $A_{w}:=\{0,1\}$ and the associated choice function $F_{w}^{2}(h^{1}):=1$ if $ (\exists g\in w)(G(g)=h(0))$, and zero otherwise.
For $F^{2}$ as in $\rado(\N^{\N})$, we have the following implications for any $n\in \N$ and where $\widetilde{n}:=\langle n\rangle* \langle n\rangle *\dots$ is a sequence:
\begin{align}\label{teasy}
(\exists g^{1})(G(g)=n)
&\di (\exists w_{0}^{1^{*}})(F_{w_{0}}(\widetilde{n})=1)\notag\\
&\di F(\widetilde{n})=1\notag\\
&\di (\exists g^{1})(G(g)=n).
\end{align}
The first implication in \eqref{teasy} follows by definition, while the others follow by the properties of $F^{2}$.
Hence, $\RANGE$ follows, yielding $\BOOT$ by \cite{samph}*{Theorem 3.19}.
\end{proof}
%
The previous proof does not make use of countable choice or $\Delta\textsf{-CA}$.  
Thus, for larger collections indexed by subsets of type $n$ objects, one readily obtains e.g.\ \ref{myhunt1} as in Corollary \ref{proofofconcepts}, but
without using extra choice or comprehension.   Finally, a reversal in Theorem \ref{dog} seems to need $\BOOT$ plus choice.

\smallskip

Hirst introduces a version of the Rado selection lemma in \cite{hirstphd}*{\S3} involving a bounding function, resulting in a reversal to $\WKL_{0}$.  
A similar bounding function could be introduced, restricting $\N^{\N}$ to some compact sub-space while obtaining (only) the Heine-Borel theorem for uncountable covers as in $\HBU$.  

\subsection{Fields and order}\label{forder}
We lift the following implication to higher-order arithmetic: \emph{that any countable formally real field is orderable implies weak K\"onig's lemma} (see \cite{simpson2}*{IV.4.5}).
This result is based on a recursive counterexample by Ershov from \cite{erjov}, as (cheerfully) acknowledged in \cite{fried4}*{p.\ 145}.

\smallskip

First of all, the aforementioned implication is obtained via an intermediate principle about the separation of disjoint ranges of injective functions (\cite{simpson2}*{IV.4.4}).  The generalisation to higher-order arithmetic and type 2 functionals is:
\be\tag{$\SEP^{1}$}
(\forall Y^{2}, Z^{2})\big[(\forall f^{1}, g^{1})(Y(f)\ne Z(g))\di (\exists X^{1})(\forall g^{1})(Y(g)\in X \wedge Z(g)\not\in X ) \big].
\ee
Modulo $\QFAC^{0,1}$, $\SEP^{1}$ is equivalent to the Heine-Borel theorem for uncountable covers as in $\HBU$.
We also need the following standard algebra definitions.
\bdefi[Field over Baire space]\label{polikoooo}
A field $K$ over Baire space consists of a set $|K|\subseteq \N^{\N}$ with distinguished elements $0_{K}$ and $1_{K}$, an equivalence relation $=_{K}$, and operations $+_{K}$, $-_{K}$ and $\times_{K}$ on $|K|$ satisfying the usual field axioms.  
\edefi
\bdefi
A field $K$ is \emph{formally real} if there is no sequence $c_{0}, \dots, c_{n}\in |K|$ such that $0=_{K}\sum_{i=0}^{n}c_{i}^{2}$.
\edefi
\bdefi
A field $K$ over $\N^{\N}$ is \emph{orderable} if there exists an binary relation `$<_{K}$' on $|K|$ satisfying the usual axioms of ordered field.
\edefi
As in \cite{simpson2}, we sometimes identify $K$ and $|K|$.
With these definitions, the following theorem is a generalisation of \cite{simpson2}*{IV.4.5.2}.
\bdefi[$\ORD$]
A formally real field over $\N^{\N}$ is orderable.
\edefi
Given the use of binary relations in $\ORD$, we shall only study the latter in the presence of $(\exists^{2})$ to guarantee that `$\leq_{\R}$' is also a binary relation as in Definition \ref{strijker}.
We now have the following theorem where the proof is obtained by lifting the proof of \cite{simpson2}*{IV.4.5} to higher-order arithmetic. 
\begin{thm}\label{hagio}
The system $\ACAo+\Delta\text{-}\CA$ proves that $\ORD\di \SEP^{1}$.
\end{thm}
\begin{proof}
Let $p_{k}$ be an enumeration of the primes and fix some $Y^{2}, Z^{2}$ as in the antecedent of $\SEP^{1}$.
By \cite{simpson2}*{II.9.7}, the algebraic closure of $\Q$, denoted $\overline{\Q}$, is available in $\RCA_{0}$.  For $w^{1^{*}}$, define $K_{w}$ as the sub-field of $\overline{\Q}(\sqrt{-1})$ generated by the following:
\[\textstyle
\{\sqrt[4]{p_{Y(w(i))}}:i<|w|\}\cup \{\sqrt{-\sqrt{  p_{Z(w(j))}}}:j<|w|\}\cup \{ \sqrt{p_{k}}:k<|w| \}.
\]
Note that one can define such a sub-field from a \emph{finite} set of generators in $\RCA_{0}$ (see \cite{simpson2}*{IV.4}). 
Unfortunately, this is not possible for \emph{infinite} sets and we need a different approach, as follows.
By the proof of Theorem \ref{rage} (and \eqref{taff} in particular), there is $G^{2}$ with: 
\be\label{conq2}
\big(\forall b\in \overline{\Q}(\sqrt{-1})\big)\big[ (\exists w^{1^{*}})(b\in K_{w})\asa (\exists v^{1^{*}})(G(v)=b)  \big].
\ee
Intuitively, we now want to define the field $\cup_{f\in \N^{\N}}K_{\langle f\rangle}$, but the latter cannot be (directly) defined as a set in weak systems. 
We therefore take the following approach: we define a field $K$ over Baire space using $G$ from \eqref{conq2}, as follows:  for $w^{1^{*}}, v^{1^{*}}$, define $w+_{K}v$ as that $u^{1^{*}}$ such that $G(u)=G(v)+_{\overline{\Q}(\sqrt{-1})}~G(w)$.
This $u^{1^{*}}$ is found by removing from $v*w$ all elements from $G(v)$ and $G(w)$ that sum to $0$ in $G(v)+_{\overline{\Q}(\sqrt{-1})}G(w)$.
Multiplication $\times_{K}$ is defined similarly, while $-_{K}w^{1^{*}}$ provides an extra label such that $G(-_{K}w)=-b$ if $G(w)=b$ and the `inverse function' of $\times_{K}$ is defined similarly.  
Using \eqref{conq2} for $w=\langle\rangle$, $0_{K}$ and $1_{K}$ are given by those finite sequences $v_{0}$ and $v_{1}$ such that $G(v_{0})=0_{\overline{\Q}(\sqrt{-1})}$ and $G(v_{1})=1_{\overline{\Q}(\sqrt{-1})}$. 
Finally, $v^{1^{*}}=_{K} w^{1^{*}}$ is defined as the decidable equality $G(w)=_{\overline{\Q}(\sqrt{-1})}G(v)$.

\smallskip

We call the resulting field $K$ and proceed to show that it is formally real.  To this end, note that $K_{w}$ can be embedded into $\overline{\Q}$ by mapping $\sqrt{p_{Y(w(i))}}$ to $\sqrt{p_{Y(w(i))}}$ and $\sqrt{p_{k}}$ to $-\sqrt{p_{k}}$ for $k\ne Y(w(i))$ for $i<|w|$.  Hence, $K_{w}$ is formally real for every $w^{1^{*}}$.
As a result, $K$ is also formally real because a counterexample to this property would live in $K_{v}$ for some $v^{1^{*}}$.  Applying $\ORD$, $K$ now has an order $<_{K}$.
Since $\sqrt{p_{Y(f)}}$ has a square root in $K_{\langle f\rangle}$, namely $\sqrt[4]{p_{Y(f)}}$, we have $u^{1^{*}}>_{K}0_{K}$ if $G(u)=\sqrt{p_{Y(f)}}$, using the basic properties of the ordered field $K$.  
One similarly obtains $v^{1^{*}}<_{K}0_{K}$ if $G(v)=\sqrt{p_{Z(g)}}$.  Intuitively speaking, the order $<_{K}$ thus allows us to separate the ranges of $Y$ and $Z$.     
To this end, consider the following equivalence, for every $k^{0}$:
\be\label{foos}
(\exists u^{1^{*}})(u>_{K} 0_{K}\wedge  G(u)=\sqrt{p_{k}})\asa (\forall v^{1^{*}})\big(  G(v)=\sqrt{p_{k}}\di  v>_{K}0_{K}\big).
\ee
The forward direction in \eqref{foos} is immediate in light of the properties of $=_{K}$ and $<_{K}$.  For the reverse direction in \eqref{foos}, fix $k_{0}$ and find $w_{0}^{1^{*}}$ such that $|w_{0}|>k_{0}$.  
Since $\sqrt{p_{k_{0}}}\in K_{w_{0}}$, \eqref{conq2} yields $v_{0}^{1^{*}}$ such that $G(v_{0})=\sqrt{p_{k_{0}}}$.  The right-hand side of \eqref{foos} implies $v_{0}>_{K}0_{K}$, and the left-hand side of \eqref{foos} follows. 

\smallskip

Finally, apply $\Delta$-comprehension to \eqref{foos} and note that the resulting set $X$ satisfies $(\forall f^{1})(Y(f)\in X\wedge Z(f)\not\in X)$.
Indeed, fix $f_{1}^{1}$ and put $k_{1}:=Y(f_{1})$.  Clearly, $\sqrt{p_{k_{1}}}\in K_{\langle f_{1}\rangle}$, yielding $v_{1}$ such that $G(v_{1})=\sqrt{p_{k_{1}}}$ by \eqref{conq2}.
As noted above, the latter number has a square root, implying $v_{1}>_{K}0$, and $Y(f_{1})=k_{1}\in X$ by definition.
Similarly, $k_{2}:=Z(f_{1})$ satisfies $\sqrt{p_{k_{2}}}\in K_{\langle f_{1}\rangle}$ and \eqref{conq2} yields $v_{2}$ such that $G(v_{2})=\sqrt{p_{k_{2}}}$.  Since $-\sqrt{p_{k_{2}}}$ has a square root in $K$, $v_{2}<_{K}0_{K}$ follows, 
and $k_{2}\not \in X$, again by the definition of $X$. 
\end{proof}
It should be noted that the field $K$ from the previous proof is actually countable (in the usual set theoretic sense).  Indeed, $K$ is can be viewed as a sub-field of $\overline{\Q}(\sqrt{-1})$.  
Moreover, there is an obvious mapping $H$ of $K$ into $\N$ such that $[H(w)=_{\overline{\Q}(\sqrt{-1})}H(v)]\di v=_{K}w$, i.e.\ an injection relative to the field equality $=_{K}$.
With this (standard) definition of `countable', we may restrict $\ORD$ to such fields and still obtain $\SEP^{1}$.

\smallskip

Next, the following theorem yield further motivation for $\Delta\text{-}\CA$.
\begin{thm}\label{motich}
The system $\RCAo$ proves $\SEP^{1}\di \Delta\text{-}\CA$.
\end{thm}
\begin{proof}
To establish this implication, let $G^{2}, H^{2}$ be such that
\be\label{hord}
(\forall k^{0})\big[(\exists f^{1})(G(f)=k)\asa (\forall g^{1})(H(g)\ne k)\big].
\ee
By definition, $G, H$ satisfy the antecedent of $\SEP^{1}$.  Let $X$ be the set obtained by applying the latter and consider:
\[
(\exists f^{1})(G(f)=k)\di k\in X\di (\forall g^{1})(H(g)\ne k)\di (\exists f^{1})(G(f)=k), 
\]
where the final implication follows from \eqref{hord}; for the special case \eqref{hord}, $X$ is now the set required by $\Delta$-$\CA$.  
For $Y^{(1\times 0)\di 0}$, define $G^{2}$ as follows for $n^{0}$ and $g^{1}$: put $G(\langle n\rangle *g)=n+1$ if $Y(g, n)=0$, and $0$ otherwise.
Note that for $k\geq 0$, we have
\[
(\exists f^{1})(Y(f,k)=0)\asa (\exists g^{1})(G(g)=k+1).
\]
Hence, \eqref{hord} is `general enough' to obtain full $\Delta$-$\CA$, and we are done.
\end{proof}
Finally, it goes without saying that one can obtain the higher-order counterpart of $\SEP^{1}$ from $\ORD$ generalised to sub-sets of $\mathcal{N}$, while
`even larger' sets give rise to the existence of the range of functionals with `even larger' domain. 



\subsection{Algebraic closure of fields}\label{aclo}
We lift two well-known RM-results pertaining to algebraic closures to higher-order arithmetic. 
Section \ref{ALA1} treats a theorem at the level of $\ACA_{0}$, while Section \ref{ALA2} treats a theorem at the level of $\WKL_{0}$.
Additional results can be found in Section \ref{moar}.
\subsubsection{At the level of arithmetical comprehension}\label{ALA1}
We lift the following implication to higher-order arithmetic: \emph{$\range$ follows from the statement that any countable field is isomorphic to a sub-field of an algebraically closed countable field} (see \cite{simpson2}*{III.3}).
This implication goes back to a recursive counterexample from \cite{vrolijk}.

\smallskip

Recall the definition of fields over $\N^{\N}$ from Definition \ref{polikoooo}.
Polynomial rings and algebraic closures are now  defined as in \cite{simpson2}*{II.9} using the conventional definitions.
Thus, the following definition makes sense in $\RCAo$.
\bdefi[$\ALCL$]
Every field over Baire space is isomorphic to a sub-field of an algebraically closed field over Baire space. 
\edefi
As above, we study $\ALCL$ in the presence of $(\exists^{2})$ to guarantee that there are `enough' binary relations. 
%
%
\begin{thm}\label{firstcome}
The system $\ACAo+\Delta\text{-}\CA$ proves $\ALCL\di \BOOT$.
\end{thm}
\begin{proof}
The proof is similar to that of Theorem \ref{hagio}, and we will not go into as much detail.  
%
Recall that $\overline{\Q}$, the algebraic closure of $\Q$, is available in $\RCA_{0}$ by \cite{simpson2}*{II.9.7} as a countable field.
Let $p_{n}$ be an increasing enumeration of the primes, i.e.\ $p_{0}=2$, $p_{1}=3$, et cetera.
Fix $Y^{2}$ and for $w^{1^{*}}$ define $K_{w}$ as the sub-field $\Q(\{\sqrt{ p_{Y(w(i))} }: i<|w|  \})$ of $\overline{\Q}$.
Similar to \eqref{conq2}, there is $G^{2}$ satisfying the following:
\be\label{conq}
(\forall b\in \overline{\Q})\big[ (\exists w^{1^{*}})(b\in K_{w})\asa (\exists v^{1^{*}})(G(v)=b)  \big].
\ee
Now define a field $K$ over Baire space using $G$ from \eqref{conq} in the same way as in the proof of Theorem \ref{hagio}. 
Apply $\ALCL$ for $K$ to obtain an isomorphism $h:K\di L\subset M$ from $K$ to a sub-field $L$ of an algebraically closed field $M$ over $\N^{\N}$.
It is then easy to see that for any $n^{0}$, we have
\be\label{simsma}
(\exists f^{1})(Y(f)=n)\asa (\forall b^{1})\big((b\in M\wedge b^{2}=p_{n} )\di b\in L\big).
\ee
Now apply $\Delta\text{-}\CA$ to \eqref{simsma} to obtain the set $X$ as required by $\RANGE$.
\end{proof}
As in Section \ref{forder},  the field $K$ from the previous proof is actually countable as it can be viewed as a sub-field of $\overline{\Q}$.  
We could therefore restrict $\ALCL$ to such fields without loss of generality, i.e.\ the previous theorem would still go through.  

\smallskip

Finally, it goes without saying that one can obtain $\RANGE^{1}$ from $\ALCL$ generalised to sub-sets of $\mathcal{N}$, while
`even larger' sets give rise to the existence of the range of functionals with `even larger' domain. 

\subsubsection{At the level of weak K\"onig's lemma}\label{ALA2}
We lift the following implication to higher-order arithmetic: \emph{that any countable field has a unique algebraic closure, implies weak K\"onig's lemma} (see \cite{simpson2}*{IV.5.2}).
This result is based on a recursive counterexample by Ershov from \cite{erjov}, as noted in \cite{fried4}*{p.\ 145}.

\smallskip

Our higher-order version of the aforementioned theorem is then as follows.
\bdefi[$\UACL$]
 A field over $\N^{\N}$ has a unique algebraic closure, i.e.\ if there are two algebraic closures $H_{i}:K\di K_{i}$ for $i=1,2$, then there is an isomorphism $H:K_{1}\di K_{2}$ such that $H\circ H_{1}=H_{2}$ on $K$.
\edefi

Since the proof is again quite similar to the above results, we shall be brief.
\begin{thm}
The system $\ACAo+\QFAC^{0,1}$ proves that $\UACL\di \SEP^{1}$,
\end{thm}
\begin{proof}
Let $p_{n}, Y, Z$ be as in the proof of Theorem \ref{hagio} and define $K_{w}$ as the sub-field of $\overline{\Q}$ generated by the following finite set:
\[\textstyle
\{\sqrt{p_{Y(w(i))}}:i<|w|\}\cup \{{\sqrt{  p_{Z(w(j))}}}:j<|w|\}.
\]
As in the previous proofs, there is $H_{1}^{2}$ with: 
\be\label{conq33}
\big(\forall b\in \overline{\Q}\big)\big[ (\exists w^{1^{*}})(b\in K_{w})\asa (\exists v^{1^{*}})(H_{1}(v)=b)  \big],
\ee
and let $K$ be the field over $\N^{\N}$ induced by $H_{1}$, like in the proof of Theorems~\ref{hagio} and~\ref{firstcome}. 
Now define another monomorphism $H_{2}$ where $H_{2}(v)$ is just $H_{1}(v)$, except with the coefficients of any $\sqrt{p_{Z(g)}}$ negated. 
Then $\UACL$ provides an isomorphism $h$ satisfying $h(H_{1}(v))=H_{2}(v)$ for $v\in K$.  
One readily proves $h(\sqrt{p_{Y(f)}})=\sqrt{p_{Y(f)}}$ and $h(\sqrt{p_{Z(g)}})=-\sqrt{p_{Z(g)}}$, namely just as in \cite{simpson2}*{IV.5.2}.  The set $X=\{k: h(\sqrt{p_{k}})=\sqrt{p_{k}}  \}$ is then 
clearly as required by $\SEP^{1}$.
\end{proof}
Finally, it goes without saying that one can obtain the higher-order counterpart of $\SEP^{1}$ from $\UACL$ generalised to sub-sets of $\mathcal{N}$, while
`even larger' sets give rise to the existence of the range of functionals with `even larger' domain. 
\subsubsection{More results in algebra}\label{moar}
Algebraic field extensions are studied in RM in \cite{dddorairs, dddorairs2} and we lift some of these results to higher-order RM.  
Note that \cite{dddorairs} is published as \cite{dddorairs2}, but we use the numbering from the former. 

\smallskip

First of all, we study \cite{dddorairs}*{Theorem 9} which establishes the equivalence between $\WKL_{0}$ and the following theorem restricted to countable fields. 
\bdefi[$\AUTO$]
Let $F$ be a field over $\N^{\N}$ with an algebraic closure $\overline{F}$. If $\alpha\in \overline{F}$ and $\varphi:F(\alpha)\di F(\alpha)$ is an $F$-automorphism of $F(\alpha)$, then $\varphi$ extends to an $F$-automorphism of $\overline{F}$.
\edefi
The notions \emph{algebraic extension} and \emph{$F$-automorphism} are defined as expected, namely as follows.  These definitions are taken from \cite{dddorairs}.
\bdefi[$\RCAo$]
An \emph{algebraic extension} of a field $F$ over $\N^{\N}$ is a pair $\langle K, \phi\rangle$, where $K$ is a field over $\N^{\N}$, $\phi$ is an embedding of $F$ into $K$, and for every $a\in K$ there is a nonzero $f(x)\in F[x]$  such that $\phi(f)(a)=0$. 
\edefi
When appropriate, we drop the mention of $\phi$ and denote the extension by $K$.
\bdefi[$\RCAo$] 
Suppose $\langle K,\phi \rangle$ and $\langle J, \psi \rangle $ are algebraic extensions of $F$. 
We say \emph{K is embeddable} in $J$ over $F$ (and write $K \preceq_{F}J$) if there is an embedding $\tau: K\di J$ such that for all $x \in F$, $\tau(\phi(x)) = \psi(x)$. 
We also say that `$\tau$ fixes $F$' and call $\tau$ an `$F$-embedding'. If $\tau$ is also bijective, we say $K$ is isomorphic to $J$ over $F$, write $K\cong_{F} J$, and call $\tau$ an $F$-isomorphism.
\edefi
In line with the previous definitions, if $K$ is an algebraic extension of $F$ that is algebraically closed, we say $K$ is an algebraic closure of $F$, and often write $\overline{F}$ for $K$.
We also need to (carefully) define the notion of extension and restriction.  
\bdefi[$\RCAo$]
Suppose $\tau: F\di G$ is a field embedding for fields $F, G$ over $\N^{\N}$, $\langle K, \phi\rangle$ is an extension of $F$, $\langle H, \psi \rangle $ is an extension of $G$, and $\theta:K\di H$ satisfies $\theta(\phi(v))=\psi(\tau(v))$ for all $v\in F$. Then we say: $\theta$ extends $\tau$, or: $\theta$ is an extension of $\tau$, or: $\theta$ restricts to $\tau$, or: $\tau$ is a restriction of $\theta$.
\edefi
With these definitions, the theorem $\AUTO$ makes sense in $\RCAo$, and we have the following theorem, obtained by lifting part of the proof of \cite{dddorairs}*{Theorem 9}.
\begin{thm}\label{contil}
The system $\ACAo$ proves $\AUTO\di \SEP^{1}$.
\end{thm}
\begin{proof}
Let $p_{n}, Y, Z$ be as in the proof of Theorem \ref{hagio}; we additionally assume that $2$ is $p_{0}$ and that $0$ is not in the range of $Y, Z$.  
The field $F_{w}=\Q(\sqrt{p_{Y(f)}}, \sqrt{2p_{Z(f)}}: f\in w^{1^{*}})$ exists in $\RCAo$, but the field $F=\Q(\sqrt{p_{Y(f)}}, \sqrt{p_{2Z(f)}}: f\in \N^{\N})$ cannot be defined directly. 
To define this field \emph{indirectly}, an equivalence similar to \eqref{conq2} can be used.  In fact, such a construction (for $\Sigma_{1}^{0}$-formulas defining sub-fields) is given in \cite{dddorairs}*{Lemma 3}.  
One therefore readily modifies \eqref{conq2} and \cite{dddorairs}*{Lemma 3} to define $F=\Q(\sqrt{p_{Y(f)}}, \sqrt{p_{2Z(f)}}: f\in \N^{\N})$ in $\RCAo$. 
By \cite{dddorairs}*{Lemma 5}, $\sqrt{2}\not \in F_{w}$ for any $w^{1^{*}}$, yielding $\sqrt{2}\not \in F$.  Choosing $\overline{F}=\overline{\Q}$, we have $\sqrt{2}\in \overline{F}$, and we may apply $\AUTO$ for $\varphi$ defined on $F(\sqrt{2})=\{a+ b\sqrt{2}:a,b\in F\}$ as $\varphi(a+b\sqrt{2}):= a-b\sqrt{2}$.  
The resulting $\overline{\varphi}$ is an automorphism on $\overline{F}$ that fixes $F$.   Now define $X:=\{n: \overline{\varphi}(\sqrt{p_{n}})=\sqrt{p_{n}}\}$ and note that this set separates the ranges of $Y, Z$ as follows:  
Since $\sqrt{p_{Y(f)}}\in F$ for any $f\in \N^{\N}$, we have $\overline{\varphi}(\sqrt{p_{Y(f)}})=\sqrt{p_{Y(f)}}$, and hence $(\exists f)(Y(f)=n)\di n\in X$.
Since $\sqrt{2p_{Z(g)}}\in F$ for any $g\in \N^{\N}$, we have $\overline{\varphi}(\sqrt{2p_{Z(g)}})=\sqrt{2p_{Z(g)}}=\sqrt{2}\sqrt{p_{Z(g)}}$ on one hand.  
On the other hand, $\overline{\varphi}(\sqrt{2p_{Z(g)}})=\overline{\varphi}(\sqrt{2})\overline{\varphi}(\sqrt{p_{Z(g)}}) =-\sqrt{2}\overline{\varphi}(\sqrt{p_{Z(g)}})$, as $\overline{\varphi}$ is a homeomorphism extending $\varphi$. 
As a result, $\overline{\varphi}(\sqrt{p_{Z(g)}})= - \sqrt{p_{Z(g)}}$, implying $Z(g)\not \in X$, as required for $\SEP^{1}$, and we are done. 
\end{proof}
\noindent
Next, let us point out two essential aspects of the previous proofs.
\begin{itemize}
\item In the countable case, finite (field) extensions can be directly defined in $\RCA_{0}$, but a `trick' is needed to similarly define infinite (field) extensions.  
The very same trick works for the uncountable case, with minimal modifications, as is clear from what follows after \eqref{conq2}.
\item The infinite (field) extension needs to have a certain property (`$\sqrt{2}\not\in F$' in the previous proof and `formally real' in the proof of Theorem \ref{hagio}).  
This property has a `compactness' flavour: if it holds for all finite (field) extensions, it holds for the infinite (field) extension.  
\end{itemize}
The previous items can perhaps inspire some kind of meta-theorem.  

\smallskip

Finally, one can similarly lift other results from \cite{dddorairs}, like the statement \emph{If $J$ is an algebraic extension of $F$, then $J$ has a root modulus}.
Indeed, the latter implies $\range$ by the proof of \cite{dddorairs}*{Theorem~17}, and one readily obtains $\RANGE$ when generalising the latter statements to fields over $\N^{\N}$.


\subsection{Rings and ideals}\label{maxiprime}
We lift the following implication to higher-order arithmetic: \emph{that any countable commutative ring has a maximal ideal, implies arithmetical comprehension} (see \cite{simpson2}*{III.5.5}).
This implication is obtained via $\range$, and the we shall obtain $\BOOT$ from the associated theorem for rings over $\N^{\N}$. 
\bdefi
A \emph{commutative ring $R$ over $\N^{\N}$} consists of a set $|R|\subseteq \N^{\N}$ with distinguished elements $0_{R}$ and $1_{R}$, an equivalence relation $=_{R}$, and operations $+_{R}$, $-_{R}$ and $\times_{R}$ on $|R|$ satisfying the usual ring axioms, including $0_{R}\ne_{R}1_{R}$.  Such a ring $R$ is an \emph{integral domain} if $(\forall s, r\in R)(r\times_{R}s=_{R}0\di r=_{R}0 \vee s=_{R}0)$.
\edefi
\bdefi
An \emph{ideal $I$} of a ring $R$ is a set $I\subseteq |R|$ such that $0_{R}\in I$ and $1_{R}\not\in I$, closed under $+_{R}$ and satisfying $(\forall a\in I,r\in R )( r \times_{R} a \in I)$.
An ideal $M$ is \emph{maximal} if $(\forall r\in R)(r\not \in M\di (\exists s\in R)((r\times_{R}s)-_{R}1\in M))$
\edefi
We again assume $(\exists^{2})$ to make sure there are `enough' binary relations.
\begin{thm}\label{morf}
Over $\ACAo$, either of the following implies $\BOOT$.
\begin{enumerate}
 \renewcommand{\theenumi}{\alph{enumi}}
\item A commutative ring over $\N^{\N}$ has a maximal ideal.
\item An integral domain over $\N^{\N}$ has a maximal ideal.\label{blag}
\end{enumerate}
\end{thm}
\begin{proof}
We show that item \eqref{blag} implies $\RANGE$.  
To this end, fix $Y^{2}$ and let $R_{0}=\Q[\langle x_{n}:n\in \N\rangle ]$ be the polynomial ring over $\Q$ with indeterminates $x_{n}$.  
Let $K_{0}=\Q(\langle x_{n}:n\in \N\rangle) $ be the associated field of fractions, i.e.\ consisting of $r/s$ where $r, s\in R_{0}$ and $s\ne_{R_{0}} 0$.
Note that $R_{0}$ and $K_{0}$ can be defined in $\RCA_{0}$ by \cite{simpson2}*{II.9}.  Now let $\varphi(b)$ be the formula expressing that $b=r/s\in K_{0} $ and there is non-empty $w^{1^{*}}, e^{0^{*}}$ and non-zero $ q\in \Q$ such that $|w|=|e|$ and some monomial in $s$ has the form $q \times x_{Y(w(0))}^{e(0)}\times \dots \times x_{Y(w(k))}^{e(k)}$ for $k=|w|-1$.
Like for \eqref{conq2} in Theorem \ref{hagio}, there is $G^{2}$ such that 
\be\label{conq3}
(\forall b \in K_{0})\big[\varphi(b) \asa (\exists v^{1^{*}})(G(v)=b) \big].
\ee 
Using $G$ as in \eqref{conq3}, define a countable integral domain $R$ by `pulling back via $G$', i.e.\ in the same way as $K$ is defined in the proof of Theorem \ref{hagio}, just below \eqref{conq2}.
Note that in particular $u^{1^{*}}=_{R}v^{1^{*}}$ if and only if $G(u)=G(v)$.
Let $M$ be the maximal ideal provided by item \eqref{blag} and consider the following for any $k\in \N$:
\be\label{kloothommel}
(\exists f^{1})(Y(f)=k)\asa (\forall v^{1^{*}})(G(v)=x_{k}\di v\not \in M).
\ee
which ends the proof by applying $\Delta\text{-}\CA$.  To prove \eqref{kloothommel}, if $(\exists f^{1})(Y(f)=k)$, then $\varphi(1/x_{k})$ and $G(v)=\frac{1}{x_{k}}$ for some $v^{1^{*}}$ by \eqref{conq3}.
Since also $\varphi(x_{k})$, any $v'$ such that $G(v')=x_{k}$ satisfies $v \times v' =_{R}1_{R}$, and hence $v'\not \in M$ as the latter is an ideal.   
Now assume the right-hand side of \eqref{kloothommel} and let us show that there is $v_{0}^{1^{*}}$ be such that $G(v_{0})=x_{k}$.  Fix $f_{0}, f_{1}$ such that $n_{0}:=Y(f_{0})$ is different from $n_{1}:=Y(f_{1})$ and define
 $b_{0}:=\frac{x_{k}x_{n_{0}}}{x_{n_{1}}}$ and $b_{1}:= \frac{x_{n_{1}}}{x_{n_{0}}}$.  Since $\varphi(b_{0})$ and $\varphi(b_{1})$, there are $a_{0}$ and $a_{1}$ such that $G(a_{0})=b_{0}$ and $G(a_{1})=b_{1}$ by \eqref{conq3}.
Then $G(a_{0}\times a_{1})=G(a_{0})\times G(a_{1})=b_{0}\times b_{1}=x_{k}$, as required.   Thus, there is $v_{0}^{1^{*}}$ such that $G(v_{0})=x_{k}$, implying $v_{0}\not \in M$ by assumption.  As $M$ is maximal, there is $a^{1^{*}}\in R$ such that $b:=(a\times v)-1_{R}$ satisfies $b\in M$.   Now, $G(b)=\frac{r}{s}$ for some $r, s\in R_{0}$ and $s\ne 0_{R}$; since $M$ is an ideal, the inverse of $b$ cannot be in $R$, i.e.\ $r$ does not contain any monomial of the aforementioned form  $q \times x_{Y(w(0))}^{e(0)}\times \dots \times x_{Y(w(k))}^{e(k)}$, while of course $s$ does by \eqref{conq3}.  Applying $G$ to $b=(a\times v)-1_{R}$, we get $\frac{r}{s}=G(b)=G(a)\times x_{k}-1_{R}$, which  implies $r+s=G(a)\times x_{k}\times s$.  By definition, no non-trivial term can divide both $r$ and $s$, so $a$ must contain $\frac{1}{x_{k}\times s}$.
There must therefore be $f^{1}$ such that $Y(f)=k$, again by definition, and \eqref{kloothommel} follows. 
 %
\end{proof}
One could obtain similar results based on the results in \cite{hatsjie}.
It goes without saying that one can obtain $\RANGE^{1}$ from the above items generalised to sub-sets of $\mathcal{N}$, while
`even larger' sets give rise to the existence of the range of functionals with `even larger' domain. 

%
%
\section{Conclusion}
We finish this paper with some conceptual and foundational remarks.
\subsection{Future work and alternative approaches}
In this section, we discuss future work and alternative approaches.  

\smallskip

First of all, it goes without saying that it is possible to obtain similar results for recursive counterexamples or reversals that are similar in kind to the ones treated above.  
Nonetheless, there are other second-order proofs that do not seem to lift to higher-order arithmetic, \emph{try as we might}.  One example is from group theory and provided by \cite{simpson2}*{III.6.5}; the proof of the latter seems 
straightforward, but so far we have been unable to provide a lifting.  The problem seems to be that we do not have any control over group generators of the kind $p_{Y(f)}$ for $Y^{2}, f^{1}$.   
One obvious question is: \emph{is there an alternative proof similar to other reversals that does lift to higher-order arithmetic}?
Other examples are from combinatorics, e.g.\ \cite{simpson2}*{III.7.5}, where the problem seems to be the lack of structure, compared to e.g.\ analysis.  
While natural numbers are `flexible' when it comes to coding (one readily switches between `the natural number' and `the objects it codes'), elements of $\R$ or $\N^{\N}$ do not have the same flexibility.  
Of course, fragments of e.g.\ Ramsey's theorem are also just false for uncountable cardinalities (see e.g. \cite{reim}). 

\smallskip

Secondly, we discuss possible alternative approaches, and why they are not fruitful. 
Now, recursive counterexamples often give rise to \emph{Brouwerian counterexamples}, and vice versa (see \cite[{p.\ xii}]{recmath1} for this opinion).  
A Brouwerian counterexample to a theorem shows that the latter is rejected in (a certain strand of) constructive mathematics (see \cite{mandje2} for details). 
We choose to use recursive counterexamples (and the associated RM results) because those are formulated in a formal system, which enables us to 
lift the associated proofs without too much trouble.  The same would not be possible for Brouwerian counterexamples, due to the lack of an explicit/unified choice of formal system for e.g.\ Bishop's constructive mathematics.  To be absolutely clear, there is nothing \emph{wrong} with constructive mathematics in general; \emph{however}, the lack of an explicit/unified formal system for constructive mathematics means that 
we cannot `lift' Brouwerian counterexamples with the same ease (or at all). 

\smallskip

Thirdly, the study of RM focuses on `ordinary' mathematics, which is usually interpreted as meaning `non-set theoretical' mathematics.
In particular, countable algebra is universally agreed to be ordinary mathematics.   
On one hand, Simpson states in \cite{takeuti}*{p.\ 432} that \emph{uncountable} algebra is not part of ordinary mathematics.  
On the other hand, the above results suggest that certain (reversal) techniques pertaining to countable algebra readily generalise to uncountable algebra, 
different in kind as the latter may be, according to Simpson.  Alternatively, one can view the results in this paper pertaining to algebra as part of \emph{countable} algebra, where the latter notion has its usual meaning.
Indeed, the field $K$ from the proof of Theorem~\ref{hagio} is countable as it can be viewed as a sub-field of $\overline{\Q}(\sqrt{-1})$.  
The same applies to the field from the proof of Theorem \ref{firstcome}.

\subsection{The bigger picture}\label{KUT}
We discuss how the above results fit into the bigger picture provided by \cites{dagsamIII, dagsamV, dagsamVI, dagsamVII, samph}.
\emph{In our opinion}, one {reasonable} interpretation of the results in this paper is that second- and higher-order arithmetic are not \emph{as} different as sometimes claimed, and that
(liftings of) recursive counterexamples and reversals provide a (partial) bridge of sorts between the two, as follows.  

\smallskip  
  
In a nutshell, Kohlenbach's higher-order RM (see Section \ref{prelim1}) is based on \emph{comprehension} and \emph{discontinuity}, while the aforementioned principles $\BOOT$, $\HBU$, and $\SEP^{1}$ cannot be captured well in this hierarchy, necessitating a new `continuity' hierarchy based on the \emph{neighbourhood function principle} $\NFP$, as first developed in \cite{samph}.
The results in this paper constitute liftings from second-order RM to this (new) $\NFP$-based hierarchy.
Let us discuss the previous claim in a lot more detail.  

\smallskip

First of all, as noted in Section \ref{prelim2}, Kohlenbach's counterpart of arithmetical comprehension $\ACA_{0}$ is given by $(\exists^{2})$, and the functional in the 
latter is clearly \emph{discontinuous}.  As shown in \cite[\S3]{kohlenbach2}, this axiom is also equivalent to e.g.\ the existence of a discontinuous function on $\R$.  
Of course, $(\exists^{2})$ follows from the Suslin functional $(\SS^{2})$ and the higher-order systems $\SIXK$ \emph{mutatis mutandis}.  
In a nutshell, Kohlenbach's higher-order comprehension is based on discontinuity. 

\smallskip

Secondly, one of the the main (conceptual) results of \cite{dagsamIII} is that higher-order comprehension does not capture e.g.\ Heine-Borel compactness well at all.  
Indeed, while $\Z_{2}^{\Omega}$ proves $\HBU$, the system $\Z_{2}^{\omega}+\QFAC^{0,1}$ does not.
The same holds for $\BOOT$ and related theorems like the Lindel\"of lemma for Baire space (see \cite{dagsamIII} for the latter). 
Since all the aforementioned theorems have relatively weak first-order strength (at most $\ACA_{0}$) in isolation, higher-order comprehension seems unsuitable for capturing them.  
\emph{By contrast}, we show in \cite{samph}*{\S5} that these axioms are equivalent to natural fragments of the \emph{neighbourhood function principle} $\NFP$, which is as follows:
\be\label{durgo}
(\forall f^{1})(\exists n^{0})A(\overline{f}n)\di (\exists \gamma\in K_{0})(\forall f^{1})A(\overline{f}\gamma(f)), 
\ee
for any formula $A\in \L_{\omega}$ and where `$\gamma\in K_{0}$' means that $\gamma^{1}$ is a total associate/RM code on $\N^{\N}$.  
Clearly, \eqref{durgo} expresses the existence of a continuous choice function.  
As it happens, $\NFP$ is a (classically valid) schema from \emph{intuitionistic analysis}, introduced in \cite{KT} and studied in \cite{troeleke1}.  

\smallskip

The previous two paragraphs merely suggest that different scales (higher-order comprehension versus $\NFP$) are better at capturing different concepts. 
The scientific enterprise is replete with examples of this state of affairs.  With the gift of hindsight, it perhaps even seems naive to believe that 
axioms consistent with Brouwer's continuity axioms (like $\HBU$ and $\BOOT$) can be captured (well) using axioms expressing discontinuity (like $\SIXK$). 
Thus, analysing third-order theorems already requires two fundamentally different scales (higher-comprehension $(\SS_{k}^{2})$ versus $\NFP$) that meet
up at $\Z_{2}^{\Omega}$, as $(\exists^{3})$ trivially yields $\SIXK$, but also $\NFP$ for any $A$ involving type 0 and 1 quantifiers only.

\smallskip

In conclusion, while comprehension is generally a great axiom schema for classifying theorems in RM (of any order), principles like $\HBU$ do not have
a nice classification based on (Kohlenbach's higher-order) comprehension.  By contrast, the latter do have a nice classification based on $\NFP$.  Now, higher-order comprehension amounts to 
\emph{discontinuity}, while $\NFP$ is a \emph{continuity} schema, stating as it does the existence of a continuous choice function.  In this light, higher-order 
arithmetic includes (at least) two `orthogonal' scales (one based on continuity, one based on discontinuity) for classifying theorems.  The `liftings' in this paper 
generalise second-order theorems to higher-order theorems from the $\NFP$-scale \emph{with little modification}. 
One disadvantage is that the `lifted' proofs from this paper often use more comprehension or countable choice than the known proofs (see Section \ref{specknetssub}).  
   
\subsection{Some technical remarks}  
In this section, we discuss the highly useful $\ECF$-interpretation and related concepts. 
\subsubsection{The $\ECF$-interpretation}\label{ECF}
The (rather) technical definition of $\ECF$ may be found in \cite{troelstra1}*{p.\ 138, \S2.6}.
Intuitively, the $\ECF$-interpretation $[A]_{\ECF}$ of a formula $A\in \L_{\omega}$ is just $A$ with all variables 
of type two and higher replaced by countable representations of continuous functionals.  Such representations are also (equivalently) called `associates' or `RM-codes' (see \cite{kohlenbach4}*{\S4}). 
The $\ECF$-interpretation connects $\RCAo$ and $\RCA_{0}$ (see \cite{kohlenbach2}*{Prop.\ 3.1}) in that if $\RCAo$ proves $A$, then $\RCA_{0}$ proves $[A]_{\ECF}$, again `up to language', as $\RCA_{0}$ is 
formulated using sets, and $[A]_{\ECF}$ is formulated using types, namely only using type zero and one objects.  

\smallskip

In light of the widespread use of codes in RM and the common practise of identifying codes with the objects being coded, it is no exaggeration to refer to $\ECF$ as the \emph{canonical} embedding of higher-order into second-order RM.  

\subsubsection{The scope of $\ECF$}
The $\ECF$ interpretation is a useful tool for gauging the strength of a given theorem.  For instance $[\HBU]_{\ECF}$ is essentially the Heine-Borel theorem for countable coverings, and the latter is equivalent to $\WKL$ by \cite{simpson2}*{IV.1}.

\smallskip

However, $[(\exists^{2})]_{\ECF}$ is `$0=1$' as $\exists^{2}$ is discontinuous on $2^{\N}$ and therefore cannot be represented by an associate. 
Hence, $\ECF$ cannot be applied to most results in this paper, at first glance, as we have mostly worked in $\ACAo$.  The truth is that 
most results, including Theorems \ref{proofofconcept}, \ref{cokkl}, \ref{dontlabel}, \ref{hagio}, \ref{firstcome}, \ref{contil} and \ref{morf}, and Corollary~\ref{rampant} also go through for $\ACAo$ replaced by $\RCAo$.  

\smallskip

To establish the previous claim, one makes use of the following fragment of the law of excluded middle: $(\exists^{2})\vee \neg(\exists^{2})$.  In the former case, one uses the above proof in $\ACAo$.  In the latter case, all functions on $\R$ and $\N^{\N}$ are continuous by \cite{kohlenbach2}*{\S3}; in this case one readily verifies that the (higher-order) conclusion always reduces to the second-order version. 
For example, the canonical covering $\cup_{x\in [0,1]}I_{x}^{\Psi}$ in $\HBU$ reduces to the countable covering $\cup_{q\in [0,1]\cap \Q}I_{q}^{\Psi}$ in case $\Psi$ is continuous. 

\smallskip

The aforementioned `excluded middle trick' was pioneered in \cite{dagsamV} and forms the basis of many results in \cites{samsplit, dagsamVII}. 
We have avoided this trick and $\ECF$ so as to provide a smoother presentation of the above material. 

\subsubsection{The nature of $\ECF$}
We discuss the meaning of the words `$A$ is converted into $B$ by the $\ECF$-translation', used throughout \cite{samph}.  
Example are: $\BOOT$ and $\HBU$ are converted to $\ACA_{0}$ and $\WKL_{0}$ by $\ECF$.

\smallskip

Such statement is obviously not to be taken literally, as $[\BOOT]_{\ECF}$ is not verbatim $\ACA_{0}$.  
Nonetheless, $[\BOOT]_{\ECF}$ follows from $\ACA_{0}$ by noting that $(\exists f^{1})(Y(f, n)=0)\asa (\exists \sigma^{0^{*}})(Y(\sigma*00, n)=0)$ for continuous $Y^{2}$ (see \cite{samph}*{\S3}).  Similarly, $[\HBU]_{\ECF}$ is not verbatim the Heine-Borel theorem for countable covers, but the latter does imply the former by noting that for continuous functions, the associated canonical cover has a trivial countable sub-cover enumerated by the rationals in $[0,1]$. 

\smallskip

In general, that continuous objects have countable representations is the very foundation of the formalisation of mathematics in $\L_{2}$, and identifying continuous objects and their countable representations is routinely done.  
Thus, when we say that $A$ is converted into $B$ by the $\ECF$-translation, we mean that $[A]_{\ECF}$ is about a class of continuous objects to which $B$ is immediately seen to apply, with a possible intermediate step involving representations.  
Since this kind of step forms the bedrock of classical RM, it would therefore appear harmless in this context.

\begin{ack}\rm
Our research was supported by the John Templeton Foundation via the grant \emph{a new dawn of intuitionism} with ID 60842.
We express our gratitude towards this institution. 
We thank Anil Nerode and Paul Shafer for their valuable advice.  
We also thank the anonymous referees for their helpful suggestions.  
Opinions expressed in this paper do not necessarily reflect those of the John Templeton Foundation.    
\end{ack}

\bye